\documentclass[reqno, 11pt]{amsart}
\usepackage{amsmath}
\usepackage{amsfonts}
\usepackage{amssymb}
\usepackage{amsthm}
\usepackage{hyperref}
\usepackage{lmodern}
\usepackage{xcolor}
\usepackage{tabularx}
\usepackage{longtable}
\usepackage{tikz}

\usepackage[text={170mm,240mm},centering]{geometry}
\input diagxy
\input xy
\xyoption{all}

\newtheorem{thm}{Theorem}[section]

\newtheorem{lem}[thm]{Lemma}
\newtheorem{prop}[thm]{Proposition}
\theoremstyle{definition}
\newtheorem{defn}[thm]{Definition}
\newtheorem{example}[thm]{Example}
\newtheorem{rem}[thm]{Remark}
\newtheorem{conj}[thm]{Conjecture}
\newtheorem{claim}[thm]{Claim}

\numberwithin{equation}{section}

\begin{document}

\title[Classification of full exceptional collections of line bundles]{Classification of full exceptional collections of line bundles on three blow-ups of $\mathbb{P}^{3}$}

\author{Wanmin Liu, Song Yang, Xun Yu}

\address{Center for Geometry and Physics, Institute for Basic Science (IBS), Pohang 37673, Republic of Korea}%
\email{wanminliu@gmail.com}%

\address{Center for Applied Mathematics, Tianjin University, Tianjin 300072, P. R. China}%
\email{syangmath@tju.edu.cn, xunyu@tju.edu.cn}%

\begin{abstract} 
A fullness conjecture of Kuznetsov says that if a smooth projective variety $X$ admits a full exceptional collection of line bundles of length $l$, then any exceptional collection of line bundles of length $l$ is full.
In this paper, we show that this conjecture holds for $X$ as the blow-up of $\mathbb{P}^{3}$ at a point, a line, or a twisted cubic curve, i.e. any exceptional collection of line bundles of length 6 on $X$ is full. Moreover, we obtain an explicit classification of full exceptional collections of line bundles on such $X$.
\end{abstract}

\date{\today}

\keywords{Derived category of coherent sheaves, full exceptional collection, semiorthogonal decomposition}

\subjclass[2010]{14F05, 14J45, 18E30}

\maketitle


\section{Introduction}

Let $X$ be a smooth projective variety and we denote by $\mathrm{D}(X)$ the bounded derived category of coherent sheaves on $X$.
It is noteworthy that $\mathrm{D}(X)$ is one of the most important invariant of $X$.
For example,  $X$ is uniquely determined by $\mathrm{D}(X)$ if the canonical bundle of $X$ is ample or anti-ample (see \cite{BO01}).
To investigate varieties via their derived categories, 
Bondal et al. \cite{BK90, BO95} introduce
the notion of {\it semiorthogonal decomposition} which has become an important tool in algebraic geometry.
In particular, the notion of semiorthogonal decomposition includes {\it full exceptional collection} as a specially important example.
More general,
any exceptional collection $\{E_1, E_2, \ldots, E_l\}$ on $\mathrm{D}(X)$ gives a semiorthogonal decomposition for $\mathrm{D}(X)$ of the form
\begin{equation}\label{exceptional-sod}
\mathrm{D}(X)=
\langle 
\mathcal{A}_{X}, E_1, E_2, \ldots,E_l\rangle,
\end{equation}
where $\mathcal{A}_{X}$ is an admissible subcategory of $\mathrm{D}(X)$ (see \cite{Bon90}).
Then $\mathcal{A}_{X}$ is trivial if and only if the exceptional collection $\{E_1,E_2,\ldots,E_l\}$ is {\it full}.
Naturally, there is a crucial problem as follows (see \cite{BO95} or \cite[Question 1.9]{Kuz-ICM}):

\textit{
Find a good condition for an exceptional collection to be full, 
or find a good condition for $\mathcal{A}_{X}$ to be trivial.
}
 
Given an admissible subcategory $\mathcal{A}_{X}\subset \mathrm{D}(X)$, 
we say that $\mathcal{A}_{X}$ is \textit{quasi-phantom} if its Hochschild homology is trivial and its Grothendieck group is of finite rank; 
moreover, it is called \textit{phantom} if the Grothendieck group is trivial.
In \cite{Kuz09}, Kuznetsov proposes a nonvanishing conjecture which asserts that 
if $\mathcal{A}_{X}$ is a quasi-phantom or phantom category 
then $\mathcal{A}_{X}$ is trivial.
However, for some surfaces $X$ of general type with $q=p_g=0$,
$\mathrm{D}(X)$ admits a semiorthogonal decomposition consisting of an exceptional collection of line bundles and its orthogonal complement which is a quasi-phantom category or phantom category (\cite{GO13, AO13, BGvBS13, GS13, Keum17, BGvBKS15, GKMS15, Lee15, LS14, KKL17,Lee16, AB17} etc). 
Motivated by those examples, the existence problem of (quasi-)phantom category on smooth projective varieties has become a much more interesting topic.

Let us now suppose that $X$ admits a full exceptional collection of length $l$.
Since the Grothendieck group is preserved under semiorthogonal decompositions,
as a consequence,  $\mathcal{A}_{X}$ in \eqref{exceptional-sod} is a phantom category.
Although Kuznetsov's nonvanishing conjecture is not true in general, 
there is still an interesting fullness conjecture attributed to Kuznetsov:
{\it any exceptional collection of length $l$ on $X$ is full} (see \cite[Conjecture 1.10]{Kuz-ICM}).
In dimension $1$, Kuznetsov's fullness conjecture is trivially true;
in dimension $2$,
a result of Kuleshov-Orlov \cite{KO94}
asserts that any exceptional collection on a del Pezzo surface is contained in a full exceptional collection and hence Kuznetsov's fullness conjecture holds for del Pezzo surfaces.

To our best knowledge, there is no common method which can be used to address Kuznetsov's fullness conjecture.
More specifically, by considering a smooth projective variety which already admits a full exceptional collection of line bundles, there is a weak version of Kuznetsov's fullness conjecture.

\begin{conj}[Kuznetsov \cite{Kuz-ICM}]\label{Kuz-icm-conj-linebundles}
If a smooth projective variety $X$ admits a full exceptional collection of line bundles of length $l$, then any exceptional collection of line bundles of length $l$ is full.
\end{conj}

This conjecture is true for $\mathbb{P}^n$ (cf. \cite{Bei78}), del Pezzo surfaces (cf. \cite{KO94}), Hirzebruch surfaces (cf. \cite{H04}) and smooth projective toric surface $X$ of Picard rank $3$ or $4$ (cf. \cite{HI13}).
To the best of our knowledge, there are no more cases supporting this conjecture.

In this paper, we show that the Conjecture \ref{Kuz-icm-conj-linebundles} holds for $X$. To be more precisely, the main result of this paper is stated as follows.

\begin{thm}[Theorem \ref{main-thm-point}+\ref{main-thm-line}+\ref{main-thm-cubic}]\label{main-thm}
Let $X$ be the blow-up of $\mathbb{P}^{3}$ at a point, or a line, or a twisted cubic curve.
Then any exceptional collection of line bundles of length $6$ on $X$ is full. Moreover, we obtain an explicit classification \footnote{ See Theorem \ref{Classify-Blowup-p-P3}, \ref{Classify-Blowup-line-P3}, \ref{Classify-Blowup-twisted-cubic-P3}. }
of full exceptional collection of line bundles on such $X$.
\end{thm}

\subsection*{Idea of Proof}
It is known that $X$ is either a $\mathbb{P}^1$-bundle over $\mathbb{P}^2$ (for the blow-up at a point or a twisted cubic curve) or $\mathbb{P}^2$-bundle over $\mathbb{P}^1$ (for the blow-up at a line). By Orlov's projective bundle formula (Theorem \ref{Orlov-projbundle-formula}), we already know a family of full exceptional collection of line bundles. However, it is not clear that any exceptional collection of line bundles could be obtained by this method. Instead, we first classify cohomologically zero line bundles. We then classify the exceptional collection of line bundles of length $6$ on $X$ up to mutations and normalizations. It turns out that in the case of blow-up a point or a twisted cubic curve, some exceptional collections of line bundles are not from the projective bundle formula. To show the fullness in the blow-up of a point case, we generalize Hille-Perling's construction of augmentation \cite{HP11} from surface case to higher dimensional case (Theorem~\ref{blowup-point-fullness-lemma}). The fullness in the blow-up of a twisted cubic case is proved by using a technical lemma that if there are two exceptional collections of the same length with only one exceptional object different, the one is full if and only if the other is full (Lemma~\ref{lem:exlb}).

The outline of this paper is organized as follows.
We devote Section \ref{preliminaries-sod} to some basic definitions and important results on semiorthogonal decompositions and full exceptional collections.
In Section \ref{Exceptional-line-bundles}, 
we give a rapid review of exceptional collection of line bundles on smooth projective varieties and mainly give a high dimensional augmentation.
In Section \ref{blow-up-one-point}, \ref{blow-up-one-line} and \ref{blow-up-one-cubic}, we study the three cases of blow-ups in Theorem \ref{main-thm}. In the last section, we give some remarks on related and further problems which we are interested.

\subsection*{Acknowledgement} 
The first author is supported by IBS-R003-D1. He thanks Kyoung-Seog Lee and Jihun Park for useful discussions. The second and third authors are partially supported by NSFC (Grant No. 11626250). They would like to thank Xiangdong Yang for giving some useful comments. The second author also thanks the IBS Center for Geometry and Physics in Pohang, South Korea, for the hospitality during his visit on August in 2017.

\subsection*{Notation and convention}
Throughout the paper, 
we work over the complex number field $\mathbb{C}$.
A variety is an integral separated schemes of finite type over $\mathbb{C}$.
Let  $X$ be a smooth projective variety,
and we will use the following notations:
\begin{center}
  \begin{tabular}{ r l }
     $\omega_X$ (resp. $K_{X}$) & the canonical bundle of $X$ (resp. the corresponding canonical divisor) \\
     $\mathrm{D}(X)$ & the bounded derived category of coherent sheaves on $X$ \\
     $\mathcal{O}_{X}(D)$ & the associated line bundle of a divisor $D$ on $X$, then $\omega_{X} =\mathcal{O}_{X}(K_X)$ \\
    $H^{i}(L)$ & $H^{i}(X, L)$ the $i$-th sheaf cohomology of line bundle $L$ \\
   $h^{i}(L)$ & $\mathrm{dim}\, H^{i}(L)$ the dimension of the $i$-th sheaf cohomology of line bundle $L$ \\
   $\mathbb{P}(\mathcal{E})$ & $\mathrm{Proj}(\mathrm{Sym\, \mathcal{E}^{\vee}})$ the projective bundle of a vector bundle $\mathcal{E}$ \\
  $f_{\ast}$ (resp. $f^{\ast}$) & the derived pushforward (resp. pullback) functor 
               of $f: X \to Y$ \\ 
               & of smooth projective varieties \\
  $\mathcal{F}\otimes \mathcal{G}$ & the derived tensor product of $\mathcal{F},  \mathcal{G}\in \mathrm{D}(X)$
  \end{tabular}
\end{center}


\section{Preliminaries}\label{preliminaries-sod}

In this section, we review some of basic definitions on semiorthogonal decompositions and gather the results that will be of importance to us later on. 
Most of the materials presented here are standard, 
we refer to \cite{Bon90, BK90, BO95, Kuz-ICM, Orlov93} for more detailed discussions.

\subsection{Semiorthogonal decompositions}

Let $\mathcal{T}$ be a triangulated category.
For any morphism $f:E\to F$ between two objects $E$ and $F$ of $\mathcal{T}$, 
there exists a distinguished triangle
$$
E \to F \to  \mathrm{Cone}(f) \to E[1],
$$
where $\mathrm{Cone}(f)$ is called the \textit{cone} of the morphism $f:E\to F$.

\begin{defn}
A \textit{semiorthogonal decomposition} of $\mathcal{T}$
is an ordered full triangulated subcategories $\{\mathcal{A}_{1},\mathcal{A}_{2},\ldots,\mathcal{A}_{l}\}$ of $\mathcal{T}$ such that
\begin{enumerate}
  \item[(1)] $\mathrm{Hom}(\mathcal{A}_{i},\mathcal{A}_{j})=0$ for $i>j$, and
  \item[(2)] for any object $T\in \mathcal{T}$, there exist $T_{i}\in \mathcal{T}$ and a sequence
$$
0=T_{l}\to T_{l-1} \to\cdots \to T_{1}\to T_{0}=T
$$
such that the cone $\mathrm{Cone}(T_{i}\to T_{i-1})\in \mathcal{A}_{i}$, for all $1\leq i\leq l$.
\end{enumerate}
For convenience, we denote by 
$$
\mathcal{T}
=\langle \mathcal{A}_{1},\mathcal{A}_{2},\ldots,\mathcal{A}_{l}\rangle,
$$
the semiorthogonal decomposition of $\mathcal{T}$ with the components $\{\mathcal{A}_{1},\mathcal{A}_{2},\ldots,\mathcal{A}_{l}\}$.
\end{defn}

To construct semiorthogonal decompositions, 
a notion of admissible category plays an important role.
Let us recall the definition.

\begin{defn}
Let $\mathcal{A}$ be a full triangulated subcategory $\mathcal{T}$.
We say that  $\mathcal{A}$ is {\it right admissible} if the inclusion functor
$i: \mathcal{A}  \hookrightarrow \mathcal{T}$ admits a right adjoint functor $i^{!}: \mathcal{T}\to \mathcal{A}$.
If  the functor $i$ has a left adjoint functor $i^{\ast}: \mathcal{T}\to \mathcal{A}$, 
then $\mathcal{A}$ is called  a {\it left admissible category}. 
We say that  $\mathcal{A}$ is {\it admissible} if $\mathcal{A}$ is both left  and right admissible. 
\end{defn}

Now, suppose $\mathcal{T}=\langle \mathcal{A},\mathcal{B}\rangle$ is a semiorthogonal decomposition. 
Then $\mathcal{A}$ is a left admissible category and $\mathcal{B}$ is a right admissible category (\cite[Lemma 3.1]{Bon90}).
Conversely, for any given left admissible category $\mathcal{A}\subset \mathcal{T}$ and right admissible category $\mathcal{B}\subset \mathcal{T}$, there exist two semiorthogonal decompositions
$$
\mathcal{T}
=\langle \mathcal{B}^{\bot},\mathcal{B}\rangle
=\langle \mathcal{A},{}^{\bot}\mathcal{A}\rangle,
$$
where 
$
\mathcal{B}^{\bot}:=\{ T \in \mathcal{T} \mid \mathrm{Hom}(B, T)=0, \forall B\in \mathcal{B}\}
$
is the right orthogonal complement of $\mathcal{B}$,
and
$
{}^{\bot}\mathcal{A}:=\{ T \in \mathcal{T} \mid \mathrm{Hom}(T, A)=0, \forall A\in \mathcal{A}\}
$
is the left orthogonal complement of $\mathcal{A}$ (\cite{BK90}).
In particular, for a semiorthogonal decomposition $\mathcal{T}=\langle \mathcal{A}_{1},\mathcal{A}_{2},\ldots,\mathcal{A}_{l}\rangle$, we have $\mathcal{A}_{j} \subset \mathcal{A}_{i}^{\bot}$ and $\mathcal{A}_{i} \subset {}^{\bot}\mathcal{A}_{j}$ for all $i>j$.
Much more specially, we have the following.

\begin{prop}[\cite{Bon90}, Theorem 3.2]
Let $X$ be a smooth projective variety.
If there exists a semiorthogonal decomposition,
$$
\mathrm{D}(X)=\langle \mathcal{A}_{1},\mathcal{A}_{2},\ldots,\mathcal{A}_{l}\rangle,
$$
then each $\mathcal{A}_{i}$ is an admissible category.
\end{prop}

\subsection{Orlov's semiorthogonal decompositions}

The first and simplest example of semiorthogonal decomposition (full exceptional collection) is Beilinson's exceptional collection (\cite{Bei78}).
Accurately, 
there is a semiorthogonal decomposition of $\mathrm{D}(\mathbb{P}^{n})$,
$$
\mathrm{D}(\mathbb{P}^{n})
=\langle
\mathcal{O}_{\mathbb{P}^{n}},\mathcal{O}_{\mathbb{P}^{n}}(H),\ldots,\mathcal{O}_{\mathbb{P}^{n}}(nH)
\rangle,
$$
where $H$ is the hyperplane of $\mathbb{P}^{n}$.
Orlov generalizes Beilinson's semiorthogonal decomposition to relative cases (\cite[Corollary 2.7]{Orlov93}).

\begin{thm}[Orlov's projective bundle formula \cite{Orlov93}]\label{Orlov-projbundle-formula}
Suppose $\mathcal{E}$ is a vector bundle of rank $r+1$ on a smooth projective variety $Y$,
and let $\rho:X:=\mathbb{P}(\mathcal{E})\to Y$ be the projective bundle of $\mathcal{E}$.
Then there is a semiorthogonal decomposition
$$
\mathrm{D}(X)=
\langle \rho^{\ast}\mathrm{D}(Y), \rho^{\ast}\mathrm{D}(Y)\otimes \mathcal{O}_{X}(1), 
\ldots,
\rho^{\ast}\mathrm{D}(Y)\otimes \mathcal{O}_{X}(r)\rangle,
$$
where $\mathcal{O}_{X}(1)$ is the Grothendieck line bundle of $X$.
\end{thm}

Let $Y$ be a smooth projective variety.
Suppose $i:Z \hookrightarrow Y$ is a closed smooth subvariety of $Y$ of codimension $r+1$.
The normal bundle of $Z$ in $Y$, denoted by $\mathcal{N}_{Z/Y}$,
is a vector bundle of rank $r+1$ on $Z$. 
The blow-up $ \mathrm{Bl}_{Z}Y$ of $Y$ at center $Z$ is a projective morphism  $\pi: \mathrm{Bl}_{Z}Y\to Y$, and the exceptional divisor of the blow-up is $E:=\pi^{-1}(Y)\cong \mathbb{P}(\mathcal{N}_{Z/Y})$. 
Then one has the following blow-up diagram
$$
\xymatrix{
E\cong \mathbb{P}(\mathcal{N}_{Z/Y}) \ar[d]_{\rho} \ar[r]^{j} & X:=\mathrm{Bl}_{Z}Y \ar[d]^{\pi}\\
 Z \ar[r]^{i} & Y.
}
$$

It is important to notice that the canonical divisor of the blow-up $X$ is determined by the following formula
$$
K_{X}=\pi^{\ast}K_{Y}+rE,
$$ 
and the restriction of $\mathcal{O}_{X}(E)$ on each fiber of $\pi$  is isomorphic to $\mathcal{O}_{\mathbb{P}^{r}}(-1)$.
Furthermore,  $\mathcal{O}_{E}(E)\cong  \mathcal{O}_{E}(-1)$.

\begin{thm}[Orlov's blow-up formula \cite{Orlov93}]\label{Orlov-blowup-formula}
Using notations as above, there is a semiorthogonal decomposition
\begin{eqnarray*}
\mathrm{D}(X)
&=&
\langle 
\pi^{\ast}\mathrm{D}(Y), j_{\ast}\rho^{\ast}\mathrm{D}(Z)\otimes \mathcal{O}_{E}(1), \ldots,
j_{\ast}\rho^{\ast}\mathrm{D}(Z)\otimes \mathcal{O}_{E}(r)
\rangle\\
&=&
\langle 
j_{\ast}\rho^{\ast}\mathrm{D}(Z)\otimes \mathcal{O}_{E}(-r), \ldots,
j_{\ast}\rho^{\ast}\mathrm{D}(Z)\otimes \mathcal{O}_{E}(-1), \pi^{\ast}\mathrm{D}(Y)
\rangle
\end{eqnarray*}
where $\mathcal{O}_{E}(1)$ is the Grothendieck line bundle of $E$.
\end{thm}

\begin{proof}
See for example \cite[Theorem 4.3]{Orlov93} or \cite[Proposition 11.18]{Huy06}.
\end{proof}

\subsection{Full exceptional collections}
This subsection reviews some of basic facts about full exceptional collections.
The full exceptional collections are the most important examples of semiorthogonal decompositions.

\begin{defn}
(1) An object $E\in \mathcal{T}$ is called an {\it exceptional object} if
$$
\mathrm{Hom}(E,E[k])
      =\left\{
      \begin{array}{ll}
      0, & k\neq 0, \\
      \mathbb{C}, & k=0.
      \end{array}
   \right.
$$

(2) An ordered sequence of exceptional objects $\{E_1,E_2,\ldots,E_{l}\}$ 
is called an {\it exceptional collection}  if 
$$
 \mathrm{Hom}(E_i,E_j[k])=0\;\;  \textrm{for any}\; i>j\, \textrm{and for all}\; k\in \mathbb{Z}.
$$
In particular, a pair $(E, F)$ is said to be an {\it exceptional pair} if the sequence $\{E, F\}$ is an exceptional collection.
Moreover, we say that the exceptional collection $\{E_1,E_2,\ldots,E_{l}\}$ is {\it strong} if 
$$
\mathrm{Hom}(E_i,E_j[k])=0,
$$
for $k\neq 0$ and for all $i, j$.

(3) An exceptional collection $\{E_1,E_2,\ldots,E_{l}\}$ is said to be {\it full} 
if the smallest full triangulated subcategory, denoted by $\langle E_1,E_2,\ldots,E_{l} \rangle$, of $\mathcal{T}$
containing $E_1,E_2,\ldots,E_{l}$ is $\mathcal{T}$ itself, i.e., $\mathcal{T}=\langle E_1,E_2,\ldots,E_{l} \rangle$, 
and then we say that $\mathcal{T}$ has a full exceptional collection of {\it length $l$}. 
\end{defn}

\begin{defn}
Let $X$ be a smooth projective variety.
We say that $X$ admits a \textit{(full) exceptional collection of length  $l$} if its derived category $\mathrm{D}(X)$ has a (full) exceptional collection of length $l$.
\end{defn}

It is not difficult to see that if $E\in \mathrm{D}(X)$ is an exceptional object 
then $\langle E\rangle\cong \mathrm{D}(\mathrm{Spec}\,\mathbb{C})$.
We see that an exceptional collection gives a splitting off copies of $\mathrm{D}(\mathrm{Spec}\,\mathbb{C})$ in $\mathrm{D}(X)$ in the sense of semiorthogonal decompositions.
Precisely, we have

\begin{prop}[\cite{Bon90}, Theorem 3.2]
Let $X$ be a smooth projective variety.
If $\{E_1,\ldots, E_l\}$ is an exceptional collection of $X$,
then there exists a semiorthogonal decomposition
$$
\mathrm{D}(X)=\langle \mathcal{A}, E_1, E_2,\ldots, E_l\rangle,
$$
where $\mathcal{A}:=\langle E_1, E_2,\ldots, E_l\rangle^{\bot}\cong\{ F\in \mathrm{D}(X) \mid \mathrm{Hom}(E_i, F[k])=0, \forall k\in \mathbb{Z}, 1\leq i\leq l\}$.
\end{prop}

For our purpose, we need the following easy but useful result.

\begin{lem}\label{usefull-fullness-lem}
If there are two exceptional collections of the same length with only one exceptional object different, then one is full if and only if the other is full.
\end{lem}

\begin{proof}
Without loss of generality, we assume $\{E_{1}, E_{2}, \cdots ,E_{n}\}$ is a full exceptional collection and $\{E_{1}', E_{2}, \ldots ,E_{n}\}$ is another exceptional collection.
Suppose $A\in \langle E_{1}', E_{2}, \ldots ,E_{n}\rangle^{\bot}$ and $A\neq 0$.
Since there is a semiorthogonal decomposition
$$
\mathrm{D}(X)=\langle E_{1}, E_{2}, \cdots ,E_{n}\rangle,
$$
and hence $\langle E_1\rangle \cong \langle E_{2}, \cdots ,E_{n} \rangle^{\bot}$.
Hence $A\in \langle E_{2}, \ldots ,E_{n} \rangle^{\bot}=\langle E_1\rangle$ 
and $A$ can be written as $A\cong \bigoplus E_1[i]^{\oplus j_i}$.
By the assumption, $(E_{1}, E_{1}')$ and $(E_{1}', E_{1})$ are not exceptional pair.
If not, the full exceptional collection $\{E_{1}, E_{2}, \cdots ,E_{n}\}$ will be extended. 
Then there some $i_0\in \mathbb{Z}$ such that $\mathrm{Hom}(E_{1}', E_1[i_0]) \neq 0$.
Therefore, we obtain 
$$
\mathrm{Hom}(\bigoplus E_{1}'[i]^{\oplus j_i}, A[i_0])
\cong 
\mathrm{Hom}(\bigoplus E_{1}'[i]^{\oplus j_i}, \bigoplus E_1[i]^{\oplus j_i}[i_0])
\neq 0.
$$
This is a contradiction and the lemma follows.
\end{proof}

\subsection{Mutations}
In this subsection, we will recall some basic facts upon the mutations which we refer to \cite{Bon90, BK90} for more details. 
Let $\mathcal{T}$ be a triangulated category.
If $i: \mathcal{A}  \hookrightarrow \mathcal{T}$ is an admissible category,
then there exist two functors, for any $F\in \mathcal{T}$,
$$
\mathrm{L}_{\mathcal{A}}(F)
:=\mathrm{Cone}(ii^{!}(F)\to F)
 \;\;\textrm{and}\;\; 
 \mathrm{R}_{\mathcal{A}}(F)
 :=\mathrm{Cone}(F\to ii^{\ast}(F))[-1],
$$
which are called the {\it left} and the {\it right mutation functors} respectively.
More specifically, if $\mathcal{A}$ is generated by an exceptional object $E$,
then
$$
\mathrm{L}_{\mathcal{A}}(F)=\mathrm{Cone}(\mathrm{RHom}(E, F)\otimes E\to F)
$$
and
$$ 
\mathrm{R}_{\mathcal{A}}(F)=\mathrm{Cone}(F\to \mathrm{RHom}(F,E)^{\ast}\otimes E)[-1].
$$

Suppose $\mathcal{T}=\langle \mathcal{A}_{1},\mathcal{A}_{2},\ldots,\mathcal{A}_{l}\rangle$ is a semiorthogonal decomposition.
Then there exist two semiorthogonal decompositions (see \cite{Bon90} or \cite[Lemma 1.9]{BK90}):
for $1\leq j \leq l-1$,
$$
\mathcal{T}
=\langle \mathcal{A}_{1},\mathcal{A}_{2},\ldots,\mathcal{A}_{j-1},\mathrm{L}_{\mathcal{A}_{j}}(\mathcal{A}_{j+1}),\mathcal{A}_{j},\mathcal{A}_{j+2},\ldots,\mathcal{A}_{l}\rangle;
$$
and for any $2\leq j \leq l$,
$$
\mathcal{T}
=\langle \mathcal{A}_{1},\mathcal{A}_{2},\ldots,\mathcal{A}_{j-2}, \mathcal{A}_{j},\mathrm{R}_{\mathcal{A}_{j}}(\mathcal{A}_{j-1}),\mathcal{A}_{j+1},\ldots,\mathcal{A}_{l}\rangle,
$$

In particular,
one has

\begin{lem}[\cite{Bon90,BK90}]\label{mutation-fullness-lem}
Let $X$ be a smooth projective variety.
If $\mathrm{D}(X)=\langle \mathcal{A},\mathcal{B} \rangle$ is a semiorthogonal decomposition, then
$\mathrm{L}_{\mathcal{A}}(\mathcal{B})=\mathcal{B}\otimes \omega_{X}$ 
and $\mathrm{R}_{\mathcal{B}}(\mathcal{A})=\mathcal{A} \otimes \omega_{X}^{\vee}$.
In particular, the sequence $\{ E_1,\ldots,E_n\}$ is a full exceptional collection
if and only if
$\{E_2,\ldots,E_n,E_{n+1}\}$
is a full exceptional collection, 
where $E_{n+1}:=E_1\otimes \omega_{X}^{\vee}$.
\end{lem}


\section{Exceptional collection of line bundles}\label{Exceptional-line-bundles}

In this section, we review some basic aspects of exceptional collection of line bundles on smooth projective varieties, which will be used in the sequel.

Let $X$ be a smooth projective variety of dimension $n$.  
A line bundle on $X$ is an exceptional object if and only if the structure sheaf $\mathcal{O}_X$ is an exceptional object. 
Recently, Sosna shows that if the Grothendieck group of $X$ is of finite rank then $\mathcal{O}_X$ is an exceptional object (\cite[Proposition 3.1]{Sosna}). 
More interesting, if $X$ is a smooth Fano variety, i.e., the canonical bundle $\omega_X$ is anti-ample, then by Kodaira vanishing theorem any line bundle on $X$ is an exceptional object. Moreover, suppose $X$ is a smooth Fano variety of Picard rank one 
and of index $r$ (i.e., $\omega_X=\mathcal{O}_{X}(-rH)$, where $H$ is the positive generator of $\mathrm{Pic}(X)$); for example, a smooth cubic fourfold is of Picard
rank one and index $3$. 
Then $r\leq n+1$ and there is a semiorthogonal decomposition of D(X),
$$
\mathrm{D}(X)=
\langle
\mathcal{A}_{X},  \mathcal{O}_{X}, \mathcal{O}_{X}(1), \ldots, \mathcal{O}_{X}(r-1)\rangle,
$$
where $\mathcal{A}_{X} = \langle \mathcal{O}_{X}, \mathcal{O}_{X}(1), \ldots, \mathcal{O}_{X}(r-1)\rangle^{\bot}$ (\cite[Corollary 3.5]{Kuz09a}). 
In particular, if $r=n+1$ then $X\cong\mathbb{P}^n$, and $\mathcal{A}_X=0$. 
Basing on a result of Bondal-Polishchuk \cite[Theorem 3.4]{BP94}, 
Vial \cite{Vial} shows that for a smooth projective variety $X$ of dimension $n$ which admits a full exceptional collection of line bundles of length $n+1$ then $X\cong\mathbb{P}^n$ (see \cite[Proposition 1.3]{Vial}).

\subsection{Basics of cohomologically zero line bundles}

Now suppose $X$ is a smooth projective variety and its structure sheaf is an exceptional object. 
In order to classify the exceptional collection of line bundles on $X$, 
we will first classify the line bundles which are cohomologically zero. 
Let us recall the definition of cohomologically zero line bundles.

\begin{defn}\label{def:cohzero}
Let $X$ be a smooth projective variety. 
A line bundle $L\in \mathrm{Pic}(X)$ is called {\it cohomologically zero} if $H^{i}(X,L)=0$ for all $i\in \mathbb{Z}$.
\end{defn}

Note that a line bundle $L$ is cohomologically zero if and only if 
the pair $(\mathcal{O}_{X}, L^{\vee})$ is an exceptional pair if and only if 
its dual line bundle $L^{\vee}$ is {\it left-orthogonal} to $\mathcal{O}_X$ in the sense of Hille-Perling \cite[Definition 3.1 (ii)]{HP11}.

\begin{example}\label{cohom-zero-F_1}
The Picard group of $\mathbb{P}^1\times\mathbb{P}^1$ is
$$
\mathrm{Pic}(\mathbb{P}^1\times\mathbb{P}^1)\cong \mathbb{Z}[S] \oplus \mathbb{Z}[F],
$$
where $S$ is the pullback of hyperplane class 
and $F$ is a fiber of $\mathbb{P}^1\times\mathbb{P}^1\to \mathbb{P}^1$. 
Then a line bundle $\mathcal{O}_{\mathbb{P}^1\times \mathbb{P}^1}(aS+bF)$ 
is cohomologically zero if and only if $a=-1$ or $b=-1$.  
\end{example}

In terms of the notion of cohomologically zero line bundles,
one has the following well-known results; see for example \cite{HP11}.

\begin{lem}\label{lem:exlb}
Let $X$ be a smooth projective variety.
Then  the sequence $\{L_1, L_2,\ldots, L_l\}$ is an exceptional collection of line bundles   on $X$ if and only if $L_j\otimes L_i^{-1}$ are cohomologically zero for all $i>j$.
\end{lem}

\begin{lem}\label{normalization-lem}
Let $X$ be a smooth projective variety.         
Then the sequence $\{L_1, L_2, \ldots, L_l\}$ is a (full) exceptional collection of line bundles if and only if the normalized sequence $\{\mathcal{O}_X, L_1^{-1}\otimes L_2, \ldots, L_1^{-1}\otimes L_l\}$ is a (full) exceptional collection.
\end{lem}

\begin{rem}
As a consequence of the above lemma,
classifying (full) exceptional collection of line bundles
is the same as classifying (full) exceptional collection of \emph{normalized} line bundles of type
$\{\mathcal{O}_X,L_2, L_3,\ldots,  L_l\}$. We also call the operation of tensoring $L_1^{-1}$ as \emph{normalization}.
\end{rem}

\subsection{Augmentation}

In the following, 
we will discuss the behavior of full exceptional collections of line bundles 
under the blow-up of a point.

To begin with, we will review the case of smooth projective surfaces.
Let $\pi: X\to Y$ be the blow-up of a smooth projective surface $Y$ at a point $p\in Y$, 
and denote by $E$ the exceptional divisor of $\pi$.
A collection of line bundles,
$$
\{\mathcal{O}_{Y}(D_1), \mathcal{O}_{Y}(D_2),\ldots, \mathcal{O}_{Y}(D_l)\}
$$ 
is a full exceptional collection if and only if, 
for any $1\leq i \leq l$, 
$$
\{\mathcal{O}_{X}(D_1+E),\ldots ,\mathcal{O}_{X}(D_{i-1}+E), \mathcal{O}_{X}(D_i), \mathcal{O}_{X}(D_i+E), \mathcal{O}_{X}(D_{i+1}), \ldots, \mathcal{O}_{X}(D_l)\}
$$
is a full exceptional collection (\cite[Proposition 2.4]{HP14}).
This process is called \textit{augmentation} (cf. \cite{HP11}).
In fact,  the augmentations allow us to construct many examples of full exceptional collections of line bundles from a known one.
In \cite{HP11}, Hille-Perling give the first systematic study of full exceptional
collections of line bundles on smooth projective surfaces and introduce the notion of {\it standard augmentation} of an exceptional collection.
Recently, Elagin-Lunts \cite{EL15} show that any full exceptional collection of line bundles on a smooth del Pezzo surface is a standard augmentation.

Next, we will generalize this augmentation from surface caes to higher dimensional case.
Now we assume $Y$ is a smooth projective variety of dimension $n$.
Let $\pi: X\to Y$ be the blow-up of $Y$ at a point $p\in Y$ with the exceptional divisor $E$.
Let 
\begin{equation}\label{blowup-point-FCE-before}
\{\mathcal{O}_{Y}(D_1), \mathcal{O}_{Y}(D_2),\ldots, \mathcal{O}_{Y}(D_l)\}
\end{equation}
be a collection of line bundles of length $l$, here $l$ must be no less than $n+1$.
Then, for any $n-1\leq i \leq l$, 
we consider the following collection
\begin{eqnarray}\label{blowup-point-FCE}
\{\mathcal{O}_{X}(D_1+(n-1)E), \ldots,\mathcal{O}_{X}(D_{i-n+1}+(n-1)E), \nonumber\\
\mathcal{O}_{X}(D_{i-n+2}+(n-2)E), \mathcal{O}_{X}(D_{i-n+2}+(n-1)E),\nonumber\\ 
\ldots,\mathcal{O}_{X}(D_{i-1}+E),  \mathcal{O}_{X}(D_{i-1}+2E), \nonumber \\
\mathcal{O}_{X}(D_{i}), \mathcal{O}_{X}(D_{i}+E), \nonumber \\
\mathcal{O}_{X}(D_{i+1}), \mathcal{O}_{X}(D_{i+2}), \ldots,  \mathcal{O}_{X}(D_{l})\}.
\end{eqnarray}

\begin{thm}\label{blowup-point-fullness-lemma}
The collection \eqref{blowup-point-FCE-before} is a full exceptional collection if and only if the collection \eqref{blowup-point-FCE} is also. 
In particular, if $\mathrm{D}(Y)$ has a full exceptional collection of line bundles, so does $\mathrm{D}(X)$.
\end{thm}

\begin{proof}
\textbf{Step 1: \eqref{blowup-point-FCE-before} is an exceptional collection if and only if \eqref{blowup-point-FCE} is an exceptional collection.}
This claim is a direct consequence from the following:
(1)  From the short exact sequence,
\begin{equation}\label{effective-exact-sequence}
0\to \mathcal{O}_{X}(-E)\to \mathcal{O}_{X}\to \mathcal{O}_{E}\to 0,
\end{equation}
we have a long exact sequence of sheaf cohomology
$$
0\to H^{0}(\mathcal{O}_{X}(-E)) \to H^{0}(\mathcal{O}_{X}) \to H^{0}(\mathcal{O}_{E})  \to H^{1}(\mathcal{O}_{X}(-E))  \to H^{1}(\mathcal{O}_{X}) \to H^{1}(\mathcal{O}_{E})\to \cdots .
$$
From this, we obtain $H^{k}(\mathcal{O}_{X}(-E))=0$ for any $k\in \mathbb{Z}$ if $\mathcal{O}_X\in D(X)$ is an exceptional object.

(2) For any $1\leq p \leq n-1$, from \eqref{effective-exact-sequence},
we tensor with $\mathcal{O}_{X}(D_j-D_i+pE)$ to obtain an exact sequence, 
$$
0\to \mathcal{O}_{X}(D_j-D_i+(p-1)E)\to \mathcal{O}_{X}(D_j-D_i+pE)\to \mathcal{O}_{E}(pE)\to 0.
$$
Since for any $1\leq p \leq n-1$, 
$\mathcal{O}_{E}(pE)\cong \mathcal{O}_{\mathbb{P}^{n-1}}(-p)$ 
is cohomologically zero on $E$, 
thus we have
$$
H^{k}(\mathcal{O}_{X}(D_j-D_i+pE))
\cong  H^{k}(\mathcal{O}_{X}(D_j-D_i+(p-1)E))\cong\cdots\cong H^{k}(\mathcal{O}_{X}(D_j-D_i)),
$$
for any $i>j$ and all $k\in \mathbb{Z}$.

\textbf{Step 2: \eqref{blowup-point-FCE-before} is full  if and only if \eqref{blowup-point-FCE} is a full.}
By Orlov's blow-up formula (Theorem \ref{Orlov-blowup-formula}),
there is a semiorthogonal decomposition of $\mathrm{D}(X)$,
\begin{equation}\label{Orlov-blowup-point-SOD}
\mathrm{D}(X)
=\langle \pi^{\ast}\mathrm{D}(Y), \mathcal{O}_{E}(E), \ldots, \mathcal{O}_{E}((n-1)E) \rangle.
\end{equation}

From  \eqref{effective-exact-sequence},
we tensor $\mathcal{O}_{X}(D_{s}+tE)$ to have an exact sequence 
\begin{equation}\label{effective-exact-sequence-twisted}
0\to \mathcal{O}_{X}(D_{s}+(t-1)E)\to \mathcal{O}_{X}(D_{s}+tE)\to \mathcal{O}_{E}(tE)\to 0.
\end{equation}
\begin{enumerate}
\item[(i)]
Suppose the collection \eqref{blowup-point-FCE-before} is a full exceptional collection.
In \eqref{effective-exact-sequence-twisted}, 
if $s=i, i-1,\ldots ,i-n+2$ and $t=i-s+1$, 
then we have
$$
0\to \mathcal{O}_{X}(D_{s}+(i-s)E)\to \mathcal{O}_{X}(D_{s}+(i-s+1)E)
\to \mathcal{O}_{E}((i-s+1)E)\to 0,
$$
and then  $\mathcal{O}_{E}(pE)\in \langle \eqref{blowup-point-FCE}\rangle$ for $1\leq p \leq n-1$.
Therefore, inductively, 
we may use the exact sequence \eqref{effective-exact-sequence-twisted} 
to show $\mathcal{O}_{X}(D_{s})\in \langle \eqref{blowup-point-FCE}\rangle$ 
for $s=1, 2, \ldots,i-1$.
For example, for $s=1$ and $1\leq p \leq n-1$, 
from the exact sequence \eqref{effective-exact-sequence-twisted}, 
we have
$$
0\to \mathcal{O}_{X}(D_{1}+(p-1)E)\to \mathcal{O}_{X}(D_{1}+pE)
\to \mathcal{O}_{E}(pE)\to 0,
$$
Therefore, if $\mathcal{O}_{X}(D_{1}+pE)\in \langle \eqref{blowup-point-FCE}\rangle$,
then $\mathcal{O}_{X}(D_{1}+(p-1)E)\in \langle \eqref{blowup-point-FCE}\rangle$ 
and thus
$\mathcal{O}_{X}(D_{1})\in \langle \eqref{blowup-point-FCE}\rangle$.

\item[(ii)]  Conversely, suppose the collection \eqref{blowup-point-FCE} is a full exceptional collection.
Assume $A\in \langle \eqref{blowup-point-FCE-before}\rangle^{\bot}$.
Next we want to show $\pi^{\ast}A \in \langle \eqref{blowup-point-FCE}\rangle^{\bot}$ and hence $A=0$.
To this end, from \eqref{effective-exact-sequence-twisted},
there is a long exact sequence
\begin{eqnarray}\label{effective-exact-sequence-twisted-long}
&\cdots &\to \mathrm{Hom}(\mathcal{O}_{E}(tE)[j+1], \pi^{\ast}A[k])
\to \mathrm{Hom}(\mathcal{O}_{X}(D_{s}+(t-1)E)[j], \pi^{\ast}A[k]) \nonumber \\
&& \to\mathrm{Hom}(\mathcal{O}_{X}(D_{s}+tE)[j], \pi^{\ast}A[k])  \to \cdots,
\end{eqnarray}
for any $j, k\in \mathbb{Z}$.
By \eqref{Orlov-blowup-point-SOD},
for any $1\leq t\leq n-1$, 
$\mathcal{O}_{E}(tE)\in {}^{\bot}\langle \pi^{\ast}\mathrm{D}(Y)\rangle$, 
we obtain
$$
\mathrm{Hom}(\mathcal{O}_{E}(tE)[j], \pi^{\ast}A[k])=0, \; \textrm{for any}\; j, k\in \mathbb{Z}.
$$
From this and \eqref{effective-exact-sequence-twisted-long},
we have
\begin{eqnarray*}
\mathrm{Hom}(\mathcal{O}_{X}(D_{s}), \pi^{\ast}A[k])\cong \cdots
&\cong& \mathrm{Hom}(\mathcal{O}_{X}(D_{s}+(t-1)E), \pi^{\ast}A[k]) \\
&\cong&  \mathrm{Hom}(\mathcal{O}_{X}(D_{s}+tE), \pi^{\ast}A[k]),
\end{eqnarray*}
for any $k\in \mathbb{Z}$, $1\leq s\leq l$ and $1\leq t\leq n-1$.
Therefore, we obtain that $\pi^{\ast}A \in \langle \eqref{blowup-point-FCE}\rangle^{\bot}$. 
\end{enumerate}
This completes the proof.
\end{proof}


\section{Blow-up a point in $\mathbb{P}^{3}$}
\label{blow-up-one-point}

In this section, we will classify cohomologically zero line bundles and hence the exceptional collection of line bundles of length $6$ on 
the blow-up of $\mathbb{P}^3$ at a point and show these exceptional collections are full.
\subsection{Geometry of $X$}
Let $\pi: X\to \mathbb{P}^{3}$ be the blow-up of $\mathbb{P}^{3}$ at a point $p\in \mathbb{P}^{3}$, and let $E$ be the exceptional divisor.
Then $X$ is a toric smooth Fano threefold with the canonical divisor
$$
K_{X}=\pi^{\ast} K_{\mathbb{P}^{3}}+2E=-4H+2E,
$$
where $H$ is the pullback of hyperplane class in $\mathbb{P}^3$.
The Picard group of $X$ is 
$$
\mathrm{Pic}(X)\cong \mathrm{Pic}(\mathbb{P}^{3})\oplus \mathbb{Z}[E]
= \mathbb{Z}[H]\oplus \mathbb{Z}[E]
$$
with the intersection numbers  
$$
H^3=1, H^2E=0, HE^2=0, E^3=1.
$$
Let $a$ be an integer. Then $X$ is also the projective bundle
\begin{equation}\label{eq:projective bundle 1}
X\cong \mathbb{P}(\mathcal{E}) \stackrel{\rho}\to \mathbb{P}^2, \text{ with } \mathcal{E}=\mathcal{O}_{\mathbb{P}^2}(-a+1)\oplus \mathcal{O}_{\mathbb{P}^2}(-a).
\end{equation}
Moreover, since $\rho^{\ast}\mathcal{O}_{\mathbb{P}^2}(1)=[H-E]$ and
$
K_X=\rho^{\ast}(K_{\mathbb{P}^{2}} +\det(\mathcal{E}^{\vee}))
        \otimes \mathcal{O}_{X}(-2),
$
we have 
$$\mathcal{O}_X(1)=\mathcal{O}_X(aH-(a-1)E).$$

\subsection{Cohomologically zero line bundles}

\begin{lem}\label{Char-H0-H3-point}
$H^{0}(\mathcal{O}_{X}(aH+bE))=0$ if and only if $a<0$ or $a+b<0$.
Consequently, $H^{3}(\mathcal{O}_{X}(aH+bE))=0$ if and only if $a>-4$ or $a+b>-2$.
\end{lem}

\begin{proof}

First, we show that $H^{0}(\mathcal{O}_{X}(aH+bE))>0$ if and only if $a\geq 0$ and $a+b\geq0$. In fact, if $a\geq 0$ and $a+b\geq0$, then $aH+bE$ is an effective divisor.
Conversely, suppose $aH+bE$ is an effective divisor ($a, b\in \mathbb{Z}$).
Since $H$ is a nef divisor,
then the intersection number
$$
H^{2}(aH+bE)=a\geq 0.
$$
Since $H-E$ is base-point free and hence a nef divisor,
and then the intersection number
$$
(H-E)^{2}(aH+bE)=a+b\geq 0.
$$

For the second part, by Serre duality, we have
$$
H^{3}(\mathcal{O}_{X}(aH+bE))\cong H^{0}(\mathcal{O}_{X}((-4-a)H-(b-2)E)).
$$
Then the lemma follows.
\end{proof}

To classify the cohomologically zero line bundles,
we have the following observation.

\begin{lem}\label{H1H2-zero-lem-point}
For any $a, b\in \mathbb{Z}$,
$h^{1}(\mathcal{O}_{X}(aH+bE))h^{2}(\mathcal{O}_{X}(aH+bE))=0$.
\end{lem}

\begin{proof}
From the short exact sequence,
$$
0\to \mathcal{O}_{X}(-E) \to \mathcal{O}_{X} \to \mathcal{O}_{E} \to 0,
$$
we tensor with $\mathcal{O}_{X}(aH+bE)$ to obtain an exact sequence of sheaves, 
$$
0\to \mathcal{O}_{X}(aH+(b-1)E)\to \mathcal{O}_{X}(aH+bE)\to \mathcal{O}_{E}(aH+bE)\to 0.
$$
Taking sheaf cohomology, 
we obtain a long exact sequence,
$$
\cdots \to  H^{1}(\mathcal{O}_{X}(aH+(b-1)E)) \to H^{1}(\mathcal{O}_{X}(aH+bE)) 
\to H^{1}( \mathcal{O}_{E}(aH+bE)) \to \cdots.
$$
Since $H^{1}(\mathcal{O}_{X}(aH))\cong H^{1}(\mathcal{O}_{\mathbb{P}^{3}}(a))=0$ 
and $H^{1}(\mathcal{O}_{E}(aH+bE))\cong H^{1}(\mathcal{O}_{\mathbb{P}^{2}}(-b))=0$,
then we have the following inequalities
$$
0=h^{1}(\mathcal{O}_{X}(aH))
\geq h^{1}(\mathcal{O}_{X}(aH+E)) 
\geq \cdots 
\geq h^{1}(\mathcal{O}_{X}(aH+(b-1)E))
\geq h^{1}(\mathcal{O}_{X}(aH+bE))
$$
for $b\geq 0$ and $a\in \mathbb{Z}$,
and hence $H^{1}(\mathcal{O}_{X}(aH+bH))=0$ for $b\geq0$ and $a\in \mathbb{Z}$.
Therefore, if $b< 0$ and thus $2-b>0$,
by Serre duality, 
then we have
$$
H^{2}(\mathcal{O}_{X}(aH+bE))\cong H^{1}(\mathcal{O}_{X}((-a-4)H-(b-2)E))=0.
$$
This completes the proof.
\end{proof}

Now we give the characterization of cohomologically zero line bundles on the blow-up of $\mathbb{P}^3$ at a point.

\begin{prop}\label{Coh-zero-line-bundles-point}
A line bundle $\mathcal{O}_{X}(aH+bE)$ is cohomologically zero if and only if one of the following holds:
\begin{enumerate}
\item[(1)]  $a+b=-1$;
\item[(2)]  $a=-1$, $b=1$;
\item[(3)]  $a=-1$, $b=2$;
\item[(4)]  $a=-2$, $b=0$;
\item[(5)]  $a=-2$, $b=2$;
\item[(6)]  $a=-3$, $b=0$;
\item[(7)]  $a=-3$, $b=1$.
\end{enumerate}
\end{prop}

\begin{proof}

Suppose $H^{0}(\mathcal{O}_{X}(aH+bE))= H^{3}(\mathcal{O}_{X}(aH+bE))=0$.
Then $\mathcal{O}_{X}(aH+bE)$ is cohomologically zero if and only if $\chi(\mathcal{O}_{X}(aH+bE))=0$.
In fact, by Lemma \ref{H1H2-zero-lem-point}, 
$h^{1}(\mathcal{O}_{X}(aH+bE))h^{2}(\mathcal{O}_{X}(aH+bE))=0$,
then we have
$$
\chi(\mathcal{O}_{X}(aH+bE))^2=(h^{1}(\mathcal{O}_{X}(aH+bE))^2+(h^{2}(\mathcal{O}_{X}(aH+bE))^2.
$$
Then $\chi(\mathcal{O}_{X}(aH+bE))=0$ if and only if $h^{1}(\mathcal{O}_{X}(aH+bE))=h^{2}(\mathcal{O}_{X}(aH+bE)=0$.

By Riemann-Roch Theorem, 
we have
\begin{eqnarray}\label{RR-formula}
\chi(\mathcal{O}_{X}(aH+bE))
&=&\int_{X} \mathrm{ch}(\mathcal{O}_{X}(aH+bE))\mathrm{Td}(X) \nonumber\\
&=& \frac{c_{1}(X)c_{2}(X)}{24}+\frac{(c_{1}^{2}(X)
        +c_{2}(X))c_{1}(\mathcal{O}_{X}(aH+bE))}{12} \nonumber\\
&&     +\frac{c_{1}(X)c_{1}^{2}(\mathcal{O}_{X}(aH+bE))}{4}
         +\frac{c_{1}^{3}(\mathcal{O}_{X}(aH+bE))}{6}.
\end{eqnarray}
By the blow-up formula of Chern classes, 
$$
c_{2}(X)=\pi^{\ast}c_{2}(\mathbb{P}^{3})=6H^2,
$$
and by Riemann-Roch formula \eqref{RR-formula},
we obtain
\begin{eqnarray*}
\chi(\mathcal{O}_{X}(aH+bE))
          &=&\frac{1}{6}(a^3+6a^2+11a+b^3-3b^2+2b+6) \\
          &=&\frac{1}{6}((a+1)(a+2)(a+3)+b(b-1)(b-2)).
\end{eqnarray*}
Therefore, by Lemma \ref{Char-H0-H3-point},
$\mathcal{O}_{X}(aH+bE)$ is cohomologically zero if and only if
the following condition hold:
\begin{enumerate}
   \item[(i)]   $a<0$, or $a+b<0$;
   \item[(ii)]   $a>-4$, or $a+b>-2$;
   \item[(iii)] $(a+1)(a+2)(a+3)+b(b-1)(b-2)=0$.
\end{enumerate}

Next, we denote $x:=a+1$ and $y:=-b$,
then $(a+1)(a+2)(a+3)+b(b-1)(b-2)=0$ is equivalent to 
$x(x+1)(x+2)=y(y+1)(y+2)$.
Since
\begin{eqnarray*}
2(x(x+1)(x+2)-y(y+1)(y+2)) 
     &=&  (x-y)(x^2+xy+y^2+3x+3y+2)\\
     &=&  (x-y)((x+y)^2+(x+3)^2+(y+3)^2-14),
\end{eqnarray*}
and hence if $x\neq y$ and $(x+y)^2+(x+3)^2+(y+3)^2=14(=1^2+2^2+3^2)$,
we have 
$$
y=0, x=-1, -2; y= -1, x=0, -1; y=-2, x=0, -1,
$$ 
i.e.,  $b=0, a=-2, -3; b=1, a=-1, -2; b=2, a=-1, -2$.
This completes the proof. 
\end{proof}

\subsection{Classification results}

\begin{thm}\label{Classify-Blowup-p-P3}
Let $X$ be the blow-up of $\mathbb{P}^{3}$ at a point.
Then the normalized sequence 
\begin{equation}\label{basic-type-of-EFC-linebundle}
\{\mathcal{O}_{X},\mathcal{O}_{X}(D_1), \mathcal{O}_{X}(D_2),\mathcal{O}_{X}(D_3), \mathcal{O}_{X}(D_4), \mathcal{O}_{X}(D_5)\}
\end{equation}
is an exceptional collection of line bundles  
if and only if the ordered set of divisors $\{D_1, D_2, D_3, D_4, D_5\}$ is one of the following types:
\begin{enumerate}
\item[$(1)_a$] $\{H-E, 2H-2E,
                aH-(a-1)E, (a+1)H-aE, 
                 (a+2)H-(a+1)E \}$;
\item[$(2)_a$] $\{H-E, aH-(a-1)E, 
                (a+1)H-aE, (a+2)H-(a+1)E,  
                3H-E\}$;
\item[$(3)_a$]  $\{aH-(a-1)E, (a+1)H-aE, 
                 (a+2)H-(a+1)E, 2H,3H-E\}$;
\item[$(4)$] $\{H-E, H,
                2H-2E, 2H-E, 3H-2E\}$;
\item[$(5)$]  $\{E, H-E, 
                H, 2H-E,  3H-E\}$;
\item[$(6)$]  $\{H-2E, H-E, 
                 2H-2E, 3H-2E,4H-3E\}$;
\item[$(7)$] $\{E, H, 
                 2H, 3H-E,3H\}$;
\item[$(8)$] $\{H-E, 2H-E, 
                3H-2E, 3H-E,  4H-3E\}$;
\item[$(9)$] $\{H, 2H-E, 
                2H, 3H-2E,  3H-E\}$,         
\end{enumerate}
where $a\in \mathbb{Z}$.
Moreover, by mutations and normalizations, they are related as:
\begin{eqnarray*}
& & (1)_a\Rightarrow (2)_a\Rightarrow (3)_a \Rightarrow (1)_{4-a} \Rightarrow  (2)_{4-a} \Rightarrow (3)_{4-a} \Rightarrow (1)_{a}; \\
& & (4)\Rightarrow  (5)\Rightarrow  (6)\Rightarrow  (7)\Rightarrow  
(8) \Rightarrow (9)\Rightarrow (4).
\end{eqnarray*}
\end{thm}

\begin{proof}
Write $D_0=0$. By Lemma \ref{lem:exlb},
the sequence \eqref{basic-type-of-EFC-linebundle}
is an exceptional collection if and only if for any integers $0\leq j<i\leq 5$ the line bundles $\mathcal{O}_{X}(D_j-D_i)$ are cohomologically zero.

Now suppose the sequence \eqref{basic-type-of-EFC-linebundle} is an exceptional collection.
First of all, by Proposition \ref{Coh-zero-line-bundles-point}, we notice that $\mathcal{O}_{X}(D_{i})$ must be one of the following line bundles:
$B_0=\mathcal{O}_{X}(aH-(a-1)E)$ ($\forall a \in \mathbb{Z}$),
$B_1=\mathcal{O}_{X}(H-E)$,
$B_2=\mathcal{O}_{X}(H-2E)$,
$B_3=\mathcal{O}_{X}(2H)$,
$B_4=\mathcal{O}_{X}(2H-2E)$,
$B_5=\mathcal{O}_{X}(3H)$,
$B_6=\mathcal{O}_{X}(3H-E)$.
Furthermore,  
to find out all the exceptional collections \eqref{basic-type-of-EFC-linebundle}, 
it suffices to pick up any five $\{B_{l_1},\cdots ,B_{l_5}\}$ such that each pair $(B_{l_j},B_{l_i})$ ($j<i$) is an exceptional pair.
In order to achieve this goal, we will build up a table of all exceptional pairs which consists of $B_i$ ($i=0,1,\cdots ,6$). 
Since a pair $(B_s, B_t)$ is an exceptional pair if and only if the line bundle  $B_s \otimes B_t^{\vee}$ is cohomologically zero, by Proposition \ref{Coh-zero-line-bundles-point} we have the following table:

\begin{table}[ht]
\caption{Exceptional pairs $(B_s, B_t)$ for blow-up of $\mathbb{P}^3$ at a point }\label{tab1:g}
\begin{center}
{\tiny
\begin{tabular}{|c| p{20mm}|c|c|c|c|c|c|}
\hline 
  & $B_0^{\prime}$ &$B_1$&$B_2$&$B_3$&$B_4$&$B_5$&$B_6$\\
\hline 
 $B_0$& $a^{\prime}=a+1$, $a+2$& $a=0$& &  $\forall a$& $a=1$&$a=0$, $1$&$\forall a$\\
\hline 
 $B_1$& $\forall a^\prime$&  & &  & $\surd$& &$\surd$\\
\hline 
$B_2$& $a^{\prime}=3$, $4$& $\surd$& & & $\surd$& & \\
\hline 
 $B_3$&$a^\prime=3$& & & &  &$\surd$&$\surd$\\
\hline
 $B_4$& $\forall a^{\prime}$&  & &  &  & & \\
\hline 
 $B_5$&  &  & &  & & & \\
\hline
 $B_6$& $a^{\prime}=4$&  & &  &  &$\surd$& \\
\hline 
\end{tabular}
}
\end{center}
\;
\;
\small{
In the table, the first column stand for $B_s$ and the first row stand for $B_t$, and the blank means that $(B_s,B_t)$ is not an exceptional pair; otherwise is.
For example, $(B_1, B_3)$ is not an exceptional pair, $(B_0, B_0')$ is an exceptional pair if and only if $a'=a+1$ or $a'=a+2$, and $(B_2, B_4)$ is an exceptional pair.}

\end{table}

Before we continue the proof,
we need the following.
 
\begin{claim}\label{small-lem1}
$\{\mathcal{O}_{X}, \mathcal{O}_{X}(a_1H+(1-a_1)E), \cdots, \mathcal{O}_{X}(a_i H+(1-a_i )E)\}$ is an exceptional collection if and only if one of the following three conditions holds:
\begin{enumerate}
\item[(1)] $i=1$, $a_1\in \mathbb{Z}$;

\item[(2)] $i=2$, $a_1=a_2-1$ or $a_1=a_2-2$;

\item[(3)] $i=3$, $a_1+1=a_2=a_3-1$.
\end{enumerate}
\end{claim}

\begin{proof}[Proof of Claim \ref{small-lem1} ]
The pair $\{\mathcal{O}_{X}(aH-(a-1)E), \mathcal{O}_{X}(bH-(b-1)E)\}$ is an exception collection if and only if $\mathcal{O}_{X}((a-b)H-(a-b)E)$ is cohomologically zero. 
Then, by Proposition \ref{Coh-zero-line-bundles-point} (2) and (5), 
we get $a-b=-1$ or $a-b=-2$ and the lemma follows.
\end{proof}

From Claim \ref{small-lem1} and Table \ref{tab1:g}, 
we observe that the line bundle $\mathcal{O}_{X}(D_1)$ in the sequence \eqref{basic-type-of-EFC-linebundle} may be only one of $B_0$, $B_1$ and $B_2$.
To finish the proof,
we shall discuss $\mathcal{O}_{X}(D_1)$ case-by-case:

\begin{enumerate}
\item[(I)] Suppose $\mathcal{O}_{X}(D_1)=B_0$.
From Table \ref{tab1:g},  for $t>0$, the exceptional pairs of type $(B_0, B_t)$ are: $(B_0, B_1)$, $(B_0,B_3)$, $(B_0, B_4)$, $(B_0, B_5)$ and $(B_0, B_6)$; then for $t_1, t_2>0$, the exceptional collections of type $\{\mathcal{O}_{X},B_0, B_{t_1}, B_{t_2}\}$ are : \underline{$\{\mathcal{O}_{X}, B_0, B_1, B_6\}$}, \underline{$\{\mathcal{O}_{X}, B_0, B_3, B_5\}$}, \underline{$\{\mathcal{O}_{X}, B_0, B_3, B_6\}$}, $\{\mathcal{O}_{X}, B_0, B_6, B_5\}$; then for $t_1, t_2, t_3>0$, the only exceptional collection of type $\{\mathcal{O}_{X}, B_0$, $B_{t_1}, B_{t_2}, B_{t_3}\}$ is :  \underline{$\{\mathcal{O}_{X}, B_0, B_3, B_6, B_5\}$}.
Then we add more divisors of type $B_0$ into the above five exceptional collections to obtain exceptional collections of length 6:
\begin{enumerate}
\item For $\{\mathcal{O}_{X},B_0, B_1,B_6\}$,  we need to add two divisors of type $B_0$, 
then we have \underline{$\{\mathcal{O}_{X},B_0, B_1$}, \underline{$B_0', B_0'', B_6\}$} (i.e., $(5)$ in Theorem \ref{Classify-Blowup-p-P3});
\item For $\{\mathcal{O}_{X}, B_0, B_3, B_5\}$, we need to add two divisors of type $B_0$, which is impossible by Table \ref{tab1:g};  

\item For $\{\mathcal{O}_{X}, B_0, B_3, B_6\}$, we need to add two divisors of type $B_0$, then we have two cases: \underline{$\{\mathcal{O}_{X}, B_0, B_0', B_3, B_0'', B_6\}$} (i.e., $(9)$ in Theorem \ref{Classify-Blowup-p-P3}), and \underline{$\{\mathcal{O}_{X}, B_0, B_0', B_0'', B_3, B_6\}$} (i.e., $(3)$ in Theorem \ref{Classify-Blowup-p-P3});

\item For $\{\mathcal{O}_{X}, B_0, B_6, B_5\}$, we need to add two divisors of type $B_0$, which is impossible by Table \ref{tab1:g};  
\item For $\{\mathcal{O}_{X}, B_0, B_3, B_6, B_5\}$, we need to add one divisor of type $B_0$, then we have \underline{$\{\mathcal{O}_{X}, B_0, B_0'$}, \underline{$B_3, B_6, B_5\}$} (i.e., $(7)$ in Theorem \ref{Classify-Blowup-p-P3}).
\end{enumerate}

\item[(II)] Suppose $\mathcal{O}_{X}(D_1)=B_1$. From Table \ref{tab1:g},  for $t> 0$, the exceptional pairs of type $(B_1, B_t)$ are: $(B_1,B_4)$, $(B_1, B_6)$; then for $t_1, t_2>0$, there are no exceptional collections of type $(B_1, B_{t_1}, B_{t_2})$.
Then we need to add divisors of type $B_0$ into the following two exceptional collections of length $3$:
\begin{enumerate}
\item  For  $(\mathcal{O}_{X}, B_1,B_4)$, we need to add  three divisors of type $B_0$, then we have two cases:  \underline{$\{\mathcal{O}_{X},B_1, B_0, B_4, B_0',B_0''\}$} (i.e., $(4)$ in Theorem \ref{Classify-Blowup-p-P3}), 
         and \underline{$\{\mathcal{O}_{X},B_1, B_4, B_0, B_0',B_0''\}$} (i.e., $(1)$ in Theorem \ref{Classify-Blowup-p-P3});
\item For  $(\mathcal{O}_{X}, B_1,B_6)$, we need to add  three divisors of type $B_0$,
         then there are also two cases: \underline{$\{\mathcal{O}_{X},B_1, B_0, B_0',B_0'', B_6\}$} (i.e., $(2)$ in Theorem \ref{Classify-Blowup-p-P3}),  
          and  \underline{$\{\mathcal{O}_{X},B_1, B_0, B_0', B_6, B_0''\}$} (i.e., $(8)$ in Theorem \ref{Classify-Blowup-p-P3}).
\end{enumerate}

\item[(III)] Suppose $\mathcal{O}_{X}(D_1)=B_2$.
From Table \ref{tab1:g}, there are at most two divisors of type $B_0$ after $D_1$ in the sequence \eqref{basic-type-of-EFC-linebundle}. On the other hand,  for $t_1, t_2>0$, there is only one exceptional collection of type $(B_2, B_{t_1}, B_{t_2})$: $(B_2, B_1, B_4)$. Thus, we get only one  exceptional collection of length 6: \underline{$\{\mathcal{O}_{X},B_2, B_1,B_4, B_0, B_0'\}$} (i.e., $(6)$ in Theorem \ref{Classify-Blowup-p-P3}).
\end{enumerate}

Recall that $\omega_X^{\vee}=\mathcal{O}_X(4H-2E)$. To check the two types of relations: 
\begin{eqnarray*}
& & (1)_a\Rightarrow (2)_a\Rightarrow (3)_a \Rightarrow (1)_{4-a} \Rightarrow  (2)_{4-a} \Rightarrow (3)_{4-a} \Rightarrow (1)_{a}; \\
& & (4)\Rightarrow  (5)\Rightarrow  (6)\Rightarrow  (7)\Rightarrow  
(8) \Rightarrow (9)\Rightarrow (4).
\end{eqnarray*}
we just use mutations (Lemma \ref{mutation-fullness-lem}) and normalizations (Lemma \ref{normalization-lem}) repeatedly. 

This completes the proof of Theorem \ref{Classify-Blowup-p-P3}.
\end{proof}

Now we are ready to give the proof of the main result of  this section.

\begin{thm}\label{main-thm-point}
Let $X$ be the blow-up of $\mathbb{P}^{3}$ at a point.
Then any exceptional collection of line bundles of length $6$ on $X$ is full.
\end{thm}

\begin{proof}

According to Lemma \ref{normalization-lem}, 
to prove Theorem \ref{main-thm-point}, 
it suffices to show that any exceptional collection of line bundles of length $6$ 
in Theorem \ref{Classify-Blowup-p-P3} is full.
By Lemma \ref{mutation-fullness-lem}, 
it suffices to show that the exceptional collections of type $(1)$ and $(4)$ in Theorem \ref{Classify-Blowup-p-P3} are full.

(i) To prove the fullness of type $(4)$ in Theorem \ref{Classify-Blowup-p-P3}, 
we will use the Beilinson's semiorthogonal decomposition of $\mathrm{D}(\mathbb{P}^3)$,
$$
\langle \mathcal{O}_{\mathbb{P}^3}, \mathcal{O}_{\mathbb{P}^3}(H), 
\mathcal{O}_{\mathbb{P}^3}(2H), \mathcal{O}_{\mathbb{P}^3}(3H)\rangle.
$$
By Theorem \ref{blowup-point-fullness-lemma}, 
we obtain a semiorthogonal decomposition of $\mathrm{D}(X)$,
$$
\langle \mathcal{O}_{X}(2E), \mathcal{O}_{X}(H+E), \mathcal{O}_{X}(H+2E), \mathcal{O}_{X}(2H), \mathcal{O}_{X}(2H+E), \mathcal{O}_{X}(3H)\rangle.
$$
Then, by Lemma \ref{normalization-lem}, this turns into the case $(4)$ in Theorem \ref{Classify-Blowup-p-P3}, and hence
$$
\mathrm{D}(X)
=\langle 
\mathcal{O}_{X}, \mathcal{O}_{X}(H-E), \mathcal{O}_{X}(H), \mathcal{O}_{X}(2H-2E), 
\mathcal{O}_{X}(2H-E), \mathcal{O}_{X}(3H-2E)
\rangle.
$$

(ii) To show the fullness of type $(1)_a$ in Theorem \ref{Classify-Blowup-p-P3}, we shall use the projective bundle structure of $X$ as (\ref{eq:projective bundle 1}). 
By Orlov's projective bundle formula (Theorem \ref{Orlov-projbundle-formula}), 
we have the following semiorthogonal decompositions
\begin{eqnarray*}
\mathrm{D}(X)
&=&\langle \rho^{\ast}\mathrm{D}(\mathbb{P}^2), 
\rho^{\ast}\mathrm{D}(\mathbb{P}^2)\otimes \mathcal{O}_X(1) \rangle\\
&=& \langle 
\rho^{\ast}\mathcal{O}_{\mathbb{P}^2}, \rho^{\ast}\mathcal{O}_{\mathbb{P}^2}(1), 
\rho^{\ast}\mathcal{O}_{\mathbb{P}^2}(2),\\
&&\;\;\;\;  \rho^{\ast}\mathcal{O}_{\mathbb{P}^2}\otimes \mathcal{O}_X(1),  
\rho^{\ast}\mathcal{O}_{\mathbb{P}^2}(1)\otimes \mathcal{O}_X(1), 
\rho^{\ast}\mathcal{O}_{\mathbb{P}^2}(2)\otimes \mathcal{O}_X(1)\rangle \\
&=& \langle  
\mathcal{O}_{X},  \mathcal{O}_{X}(H-E), \mathcal{O}_{X}(2H-2E),
 \mathcal{O}_{X}\otimes \mathcal{O}_{X}(aH-(a-1)E),  \\
 && \;\;\;\; \mathcal{O}_{X}(H-E)\otimes \mathcal{O}_{X}(aH-(a-1)E), 
\mathcal{O}_{X}(2H-2E)\otimes \mathcal{O}_{X}(aH-(a-1)E)\rangle \\
&=&  \langle 
\mathcal{O}_{X}, \mathcal{O}_{X}(H-E), \mathcal{O}_{X}(2H-2E),
\mathcal{O}_{X}(aH-(a-1)E), \\
&&\;\;\;\;  \mathcal{O}_{X}((a+1)H-aE),\mathcal{O}_{X}((a+2)H-(a+1)E) \rangle. 
\end{eqnarray*}
This completes the proof of  Theorem \ref{main-thm-point}.
\end{proof}


\section{Blow-up a line in $\mathbb{P}^3$}
\label{blow-up-one-line}

In this section, we will classify the cohomologically zero line bundles 
and the exceptional collection of line bundles of length $6$ on the blow-up of $\mathbb{P}^3$ at a line and show they are full.
\subsection{Geometry of $X$}
Let $\pi: X\to \mathbb{P}^3$ be the blow-up of $\mathbb{P}^3$ 
at a line $\mathbb{P}^1 \cong \mathbb{L} \subset \mathbb{P}^3$.
The exceptional divisor of $\pi$ is
$E\cong\mathbb{P}(\mathcal{N}_{\mathbb{P}^1/ \mathbb{P}^3})\cong\mathbb{P}^{1}\times \mathbb{P}^{1}$.
Then $X$ is a toric smooth Fano threefold with the canonical divisor
$$
K_{X}=\pi^{\ast}K_{\mathbb{P}^3}+E=-4H+E,
$$
where $H$ is the pullback of hyperplane class in $\mathbb{P}^3$.
The Picard group of $X$ is 
$$
\mathrm{Pic}(X)\cong \mathrm{Pic}(\mathbb{P}^3)\oplus \mathbb{Z}[E]=\mathbb{Z}[H]\oplus \mathbb{Z}[E]
$$
with intersection numbers
$$
H^3=1, H^2E=0, HE^2=-1, E^3=-2,
$$
and we may assume $\mathcal{O}_{E}(E)\cong \mathcal{O}_{E}(-S+F)$ (note that $(aS+bF)^2=-2$ implies $aS+bF=\pm (-S+F)$),  
$\mathcal{O}_{E}(H)\cong \mathcal{O}_{E}(F)$, where $S$ and $F$ are given in Example \ref{cohom-zero-F_1}.

Let $a$ be an integer. Then $X$ is also the projective bundle
\begin{equation}\label{eq:projective bundle 2}
X\cong \mathbb{P}(\mathcal{V}) 
\stackrel{\rho}\to \mathbb{P}^{1}, \text{ with }
\mathcal{V}:=\mathcal{O}_{\mathbb{P}^{1}}(-a+1)^{\oplus 2} \oplus \mathcal{O}_{\mathbb{P}^1}(-a).
\end{equation}
Since $\rho^{\ast}\mathcal{O}_{\mathbb{P}^{1}}(1)=[H-E]$ and
$
K_X=\rho^{\ast}(K_{\mathbb{P}^{1}}+\det(\mathcal{V}^{\vee}))
         \otimes \mathcal{O}_{X}(-3),
$
we have 
$$\mathcal{O}_X(1)=\mathcal{O}_X(aH-(a-1)E).$$
\subsection{Cohomologically zero line bundles}

\begin{lem}\label{Char-H0-H3-line}
$H^{0}(\mathcal{O}_{X}(aH+bE))=0$ if and only if $a<0$ or $a+b<0$.
Consequently, $H^{3}(\mathcal{O}_{X}(aH+bE))=0$ if and only if $a>-4$ or $a+b>-3$.
\end{lem}

\begin{proof}
Similar to Lemma \ref{Char-H0-H3-point},
it suffices to show that if $aH+bE$ ($a, b\in \mathbb{Z}$) is an effective divisor, 
then $a\geq 0$ and $a+b\geq0$.
Suppose $aH+bE$ is an effective divisor, $a, b\in \mathbb{Z}$.
Since $H$ is a nef divisor,
then the intersection number
$$
H^{2}(aH+bE)=a\geq 0.
$$
Since $H, H-E$ are base-point free and hence are nef divisors,
and then
$$
H(H-E)(aH+bE)=a+b\geq 0.
$$

By Serre duality,
we have
$$
H^{3}(\mathcal{O}_{X}(aH+bE))\cong H^{0}(\mathcal{O}_{X}((-4-a)H-(b-1)E)).
$$
It follows the proposition.
\end{proof}

Let $P\cong \mathbb{P}^2$ be the proper transform of a plane containing the line 
$\mathbb{L}$.
Then $\mathcal{O}_{X}(P)\cong\mathcal{O}_{X}(H-E)$.

\begin{lem}\label{divisor-P-lem}
$H|_{P}=\mathcal{O}_{P}(H)\cong \mathcal{O}_{P}(1)$ 
and $E|_{P}=\mathcal{O}_{P}(H)\cong \mathcal{O}_{P}(1)$.
\end{lem}

\begin{proof}
Since the divisor $H$ is nef, 
then the line bundle $\mathcal{O}_{P}(H)$ is also nef. 
Thus $\mathcal{O}_{P}(H)\cong \mathcal{O}_{P}(k)$ for some $k\geq 0$.
We obtain intersection numbers
$$
k^2=(H|_{P})^{2}=H^2(H-E)=H^3-H^2E=1,
$$
and hence $k=1$.

Suppose $E|_{P}\cong \mathcal{O}_{P}(k')$.
Then the intersection numbers
$$
k'=(E|_{P})(H|_{P})=EH(H-E)=1,
$$
and we have $k'=1$.
\end{proof}

 Analogous to Lemma \ref{H1H2-zero-lem-point},
 we have the following result.

\begin{lem}\label{H1H2-zero-lem-line}
For any $a, b\in \mathbb{Z}$, 
$h^{1}(\mathcal{O}_{X}(aH+bE))h^{2}(\mathcal{O}_{X}(aH+bE))=0$.
\end{lem}

\begin{proof}
We first show that $H^{1}(\mathcal{O}_{X}(sP+tH))=0$ if $s\geq 0$.
From the short exact sequence
$$
0 \to \mathcal{O}_{X}(-P)\to \mathcal{O}_{X}\to \mathcal{O}_{P} \to 0,
$$
we tensor $\mathcal{O}_{X}(sP+tH)$ to obtain a short exact sequence
$$
0 \to \mathcal{O}_{X}((s-1)P+tH)\to \mathcal{O}_{X}(sP+tH)\to \mathcal{O}_{P}(sP+tH) \to 0.
$$
Taking cohomology,
we have a long exact sequence
$$
\cdots \to H^1(\mathcal{O}_{X}(s-1)P+tH))\to H^1(\mathcal{O}_{X}(sP+tH))\to H^{1}(\mathcal{O}_{P}(sP+tH)) \to \cdots.
$$
Since $P\cong \mathbb{P}^{2}$, by Lemma \ref{divisor-P-lem},
we obtain $H^{1}(\mathcal{O}_{P}(sP+tH))\cong H^{1}(\mathcal{O}_{P}(t))=0$ for $t\in \mathbb{Z}$.
Then we have the following inequalities
$$
0=h^{1}(\mathcal{O}_{X}(tH))
\geq h^{1}(\mathcal{O}_{X}(P+tH)) 
\geq \cdots 
\geq h^{1}(\mathcal{O}_{X}((s-1)P+tH))
\geq h^{1}(\mathcal{O}_{X}(sP+tH)),
$$
and hence $H^{1}(\mathcal{O}_{X}(sP+tH))=0$ for $s\geq0$ and $t\in \mathbb{Z}$.

Secondly, since $aH+bE=-b(H-E)+(a+b)H$ and if $-b\geq 0$, 
then $H^{1}(\mathcal{O}_{X}(aH+bE))=0$ for $b\leq 0$ and $a\in \mathbb{Z}$.
If $-b<0$, i.e., $b\geq 1$, 
by Serre duality, we have
\begin{eqnarray*}
H^{2}(\mathcal{O}_{X}(aH+bE))
&\cong& H^{1}(\mathcal{O}_{X}((-4-a)H-(b-1)E)) \\
&=& H^{1}(\mathcal{O}_{X}((b-1)(H-E)+(-3-a-b)H))=0.
\end{eqnarray*}
This completes the proof.
\end{proof}

Next, we give the characterization of cohomologically zero line bundles on the blow-up of $\mathbb{P}^3$ at a line.

\begin{prop}\label{Coh-zero-line-bundles-line}
A line bundle $\mathcal{O}_{X}(aH+bE)$ is cohomologically zero if and only if one of the following conditions hold:
\begin{enumerate}
\item[(1)]  $a+b=-1$;
\item[(2)]  $a+b=-2$;
\item[(3)]  $a=-1$, $b=1$;
\item[(4)]  $a=-3$, $b=0$.
\end{enumerate}
\end{prop}

\begin{proof}

By Blow-up formula of Chern classes, 
we have
$$
c_2(X)
=\pi^{\ast}(c_2(\mathbb{P}^3)+\mu_{\mathbb{P}^1})-\pi^{\ast}c_{1}(\mathbb{P}^3)E
=7H^2-4HE,
$$
where $\mu_{\mathbb{P}^1}=H^2$.
By Riemann-Roch formula \eqref{RR-formula}, 
we have
\begin{eqnarray*}
\chi(\mathcal{O}_{X}(aH+bE)) &=&  \frac{1}{6}(a^3+6a^2+11a+6-2b^3-3b^2+5b+3ab-3ab^2)\\
           &=& \frac{1}{6}(a-2b+3)(a+b+1)(a+b+2).
\end{eqnarray*}
Therefore, by Lemma \ref{Char-H0-H3-line} and Lemma \ref{H1H2-zero-lem-line},
$\mathcal{O}_{X}(aH+bE)$ is cohomologically zero 
if and only if $\chi(\mathcal{O}(aH+bE))=0$ and one of the following conditions hold:
\begin{enumerate}
    \item[(1)] $-4<a<0$;
    \item[(2)] $a<0$, $a+b>-3$;
    \item[(3)] $a+b<0$, $a>-4$;
    \item[(4)] $-3<a+b<0$.
\end{enumerate}
Then the proposition follows the case by case:

(I) If $-3<a+b<0$, i.e., $a+b=-1$ or $a+b=-2$, then $\chi(\mathcal{O}_{X}(aH+bE))=0$.

(II) We assume $(a+b+1)(a+b+2)\neq 0$ and $\chi(\mathcal{O}_{X}(aH+bE))=0$ (i.e., $a-2b+3=0$).
Then we obtain:
\begin{enumerate}
    \item[(i)] if $-4<a<0$, then $(a=-1, b=1)$ and $(a=-3, b=0)$;
    \item[(ii)] if $a<0$ and $a+b>-3$, then $a=-1, b=1$;
    \item[(iii)] if $a+b<0$ and $a>-4$, then $a=-3, b=0$;
\end{enumerate}
This proves the proposition.
\end{proof}

\subsection{Classification results}

\begin{thm}\label{Classify-Blowup-line-P3}
Let $X$ be the blow-up of $\mathbb{P}^{3}$ at a line.
Then the normalized sequence 
\begin{equation}\label{basic-type-of-EFC-linebundle2}
\{\mathcal{O}_{X},\mathcal{O}_{X}(D_1), \mathcal{O}_{X}(D_2),\mathcal{O}_{X}(D_3), \mathcal{O}_{X}(D_4), \mathcal{O}_{X}(D_5)\}
\end{equation}
is an exceptional collection of line bundles  
if and only if the ordered set of divisors $\{D_1, D_2, D_3, D_4, D_5\}$ is one of the following two types:
\begin{enumerate}
\item[$(1)_{a,b}$] $\{aH-(a-1)E, (a+1)H-aE,
               bH-(b-2)E, (b+1)H-(b-1)E, 
               3H\}$;
\item[$(2)_{a,b}$] $\{H-E,aH-(a-1)E, 
               (a+1)H-aE, bH-(b-2)E, 
               (b+1)H-(b-1)E\}$,
\end{enumerate}
where $a, b\in \mathbb{Z}$. Moreover, by mutations and normalizations, they are related as
$$
(1)_{a,b}\Rightarrow (2)_{b-a,3-a}\Rightarrow (1)_{b-a-1,2-a}.
$$
\end{thm}

\begin{proof}
The idea of the proof is analogous to that of Theorem \ref{Classify-Blowup-p-P3} and it is relatively easy. Write $D_0=0$. By Lemma \ref{lem:exlb},
the sequence \eqref{basic-type-of-EFC-linebundle2}
is an exceptional collection if and only if for any integers $0\leq j<i\leq 5$ the line bundles $\mathcal{O}_{X}(D_j-D_i)$ are cohomologically zero.

Suppose the sequence \eqref{basic-type-of-EFC-linebundle2} is an exceptional colection. According to Proposition \ref{Coh-zero-line-bundles-line},
for any $1 \leq i \leq 5$,
$\mathcal{O}_{X}(D_{i})$ must be one of the line bundles:
$B_0=\mathcal{O}_{X}(aH-(a-1)E)$,
$B_1=\mathcal{O}_{X}(bH-(b-2)E)$,
$B_2=\mathcal{O}_{X}(H-E)$,
and $B_3=\mathcal{O}_{X}(3H)$,
where $a, b\in \mathbb{Z}$.
To determine all the exceptional collections of line bundle of length $6$, 
analogously, by Proposition \ref{Coh-zero-line-bundles-line}
we have the following table of exceptional pairs $(B_s, B_t)$ which consists of $B_0, B_1, B_2$ and $B_3$:
\begin{table}[ht]
\caption{Exceptional pairs $(B_s, B_t)$ for blow-up of $\mathbb{P}^3$ at a line }\label{tab2:g}
\begin{center}
{\tiny
\begin{tabular}{|c| p{12mm}|c|c|c|c|}
\hline 
  & $B_0^{\prime}$ &$B_1'$&$B_2$&$B_3$\\
\hline 
 $B_0$& $a'=a+1$ & $\forall a, b'$& & $\forall a$ \\
\hline 
 $B_1$&  & $b'=b+1$  & & $\forall b$ \\
\hline 
$B_2$& $\forall a$  & $\forall b'$ &  &  \\
\hline 
 $B_3$& & & &   \\

\hline 
\end{tabular}
}
\end{center}
\end{table}

Consequently, $\mathcal{O}_{X}(D_1)$ may be one of $B_0$ and $B_2$, 
and we have the following two cases:
\begin{enumerate}
\item if $\mathcal{O}_{X}(D_1)=B_0$, then we get the exceptional collection: 
$\{\mathcal{O}_{X},  B_0, B_0', B_1, B_1', B_3\} $ (i.e., $(1)$ in Theorem \ref{Classify-Blowup-line-P3});

\item if $\mathcal{O}_{X}(D_1)=B_2$, then we obtain the exceptional collection:
$\{\mathcal{O}_{X}, B_2, B_0, B_0', B_1, B_1' \}$ (i.e., $(2)$ in Theorem \ref{Classify-Blowup-line-P3}).
\end{enumerate}

Recall that $\omega_X^{\vee}=\mathcal{O}_X(4H-E)$. To check the relations
$$
(1)_{a,b}\Rightarrow (2)_{b-a,3-a}\Rightarrow (1)_{b-a-1,2-a},
$$
we just use mutations (Lemma \ref{mutation-fullness-lem}) and normalizations (Lemma \ref{normalization-lem}) repeatedly. 
This completes the proof of Theorem \ref{Classify-Blowup-line-P3}.
\end{proof}

Now we are ready to prove the main result of this section.

\begin{thm}\label{main-thm-line}
Let $X$ be the blow-up of $\mathbb{P}^{3}$ at a line.
Then any exceptional collection of line bundles of length $6$ 
on $X$ is full.
\end{thm}

\begin{proof}
By Lemma \ref{normalization-lem} and Lemma \ref{mutation-fullness-lem},
it suffices to show the exceptional collection $(2)$ in Theorem \ref{Classify-Blowup-line-P3} is full. 
By using the projective bundle structure of $X$ (\ref{eq:projective bundle 2}) and Orlov's projective bundle formula (Theorem \ref{Orlov-projbundle-formula}), 
we have semiorthogonal decompositions of $\mathrm{D}(X)$,
\begin{eqnarray*} 
\mathrm{D}(X)
&=&
\langle 
\rho^{\ast}\mathrm{D}(\mathbb{P}^1), 
\rho^{\ast}\mathrm{D}(\mathbb{P}^1)\otimes \mathcal{O}_X(1), 
\rho^{\ast}\mathrm{D}(\mathbb{P}^1)\otimes \mathcal{O}_X(2) 
\rangle \\
&=& 
\langle 
\rho^{\ast}\mathcal{O}_{\mathbb{P}^1}, \rho^{\ast}\mathcal{O}_{\mathbb{P}^1}(1), 
\rho^{\ast}\mathcal{O}_{\mathbb{P}^1}\otimes\mathcal{O}_X(1), \rho^{\ast}\mathcal{O}_{\mathbb{P}^1}(1)\otimes \mathcal{O}_X(1),  \\ 
&& \;\;\;\; \rho^{\ast}\mathcal{O}_{\mathbb{P}^1}\otimes \mathcal{O}_X(2), 
\rho^{\ast}\mathcal{O}_{\mathbb{P}^1}(1)\otimes \mathcal{O}_X(2)\rangle \\
&=&
\langle 
\mathcal{O}_{X}, \mathcal{O}_{X}(H-E), 
\mathcal{O}_X(aH-(a-1)E), \mathcal{O}_{X}((a+1)H-aE) , \\ 
&& \;\;\;\; \mathcal{O}_X(2aH-(2a-2)E), \mathcal{O}_{X}((2a+1)H-(2a-1)E)
\rangle.
\end{eqnarray*}
Based on this semiorthogonal decomposition, we may inductively show that the exceptional collection $(2)$ in Theorem \ref{Classify-Blowup-line-P3} is full.
For any given $a\in \mathbb{Z}$, inductively, 
we assume that for $b=k$ the exceptional collection 
\begin{eqnarray}\label{main-thm-line-equ1}
&& \{\mathcal{O}_{X}, \mathcal{O}_{X}(H-E),
      \mathcal{O}_{X}(aH-(a-1)E),  \mathcal{O}_{X}((a+1)H-aE),\nonumber \\
&& \;\;\;\; \mathcal{O}_{X}(kH+(2-k)E), \mathcal{O}_{X}((k+1)H+(1-k)E)\}
\end{eqnarray}
is full.
Then, for $b=k-1$,  we have the exceptional collection
\begin{eqnarray}\label{main-thm-line-equ2}
&& \{\mathcal{O}_{X}, \mathcal{O}_{X}(H-E),
      \mathcal{O}_{X}(aH-(a-1)E),  \mathcal{O}_{X}((a+1)H-aE),\nonumber \\
&& \;\;\;\; \mathcal{O}_{X}((k-1)H-(k-3)E), \mathcal{O}_{X}(kH-(k-2)E)\},
\end{eqnarray}
and for $b=k+1$ we have the exceptional collection
\begin{eqnarray}\label{main-thm-line-equ3}
&& \{\mathcal{O}_{X}, \mathcal{O}_{X}(H-E),
      \mathcal{O}_{X}(aH-(a-1)E),  \mathcal{O}_{X}((a+1)H-aE), \nonumber \\
&& \;\;\;\; \mathcal{O}_{X}((k+1)H-(k-1)E), \mathcal{O}_{X}((k+2)H-kE)\}.
\end{eqnarray}
Comparing the exceptional collections \eqref{main-thm-line-equ2} and \eqref{main-thm-line-equ3} with \eqref{main-thm-line-equ1},
by Lemma \ref{usefull-fullness-lem}, 
the exceptional collection \eqref{main-thm-line-equ2} and \eqref{main-thm-line-equ3} are full.
This finishes the proof.
\end{proof}


\section{Blow-up a twisted cubic curve in $\mathbb{P}^3$}
\label{blow-up-one-cubic}

In this section, we shall classify the exceptional collection of line bundles of length $6$ on 
the blow-up of $\mathbb{P}^3$ at a twisted cubic curve and show they are all full.
\subsection{Geometry of $X$}
Let $C\subset \mathbb{P}^3$ be a smooth rational curve of degree $3$ (i.e., a twisted cubic curve). 
The normal bundle of $C$ in $\mathbb{P}^3$ is $\mathcal{N}_{C/\mathbb{P}^3}\cong \mathcal{O}_{C}(5)\oplus \mathcal{O}_{C}(5)$ (see \cite[Proposition 6]{EV81}).
Let $\pi:X \to \mathbb{P}^3$ be  the blow-up of $\mathbb{P}^3$ at $C$.
Then $X$ is a (non-toric) smooth Fano threefold.
Let $E$ be the exceptional divisor of $\pi$, 
that is, $E:=\mathbb{P}(\mathcal{N}_{C/\mathbb{P}^3})\cong \mathbb{P}^1\times \mathbb{P}^1$.
We denote by $H$ the pull back of hyperplane class in $\mathbb{P}^3$,
and then the canonical divisor  of $X$ is
$$
K_{X}=\pi^{\ast}K_{\mathbb{P}^3}+E=-4H+E.
$$
The Picard group of $X$ is
$$
\mathrm{Pic}(X)\cong \mathrm{Pic}(\mathbb{P}^3)\oplus \mathbb{Z}[E]
=\mathbb{Z}[H]\oplus \mathbb{Z}[E]
$$
with intersection numbers
$$
H^3=1, H^2E=0, HE^2=-3, E^3=-10,
$$
and $\mathcal{O}_{E}(E)\cong \mathcal{O}_{E}(-S+5F)$,  
$\mathcal{O}_{E}(H)\cong \mathcal{O}_{E}(3F)$, where $S$ and $F$ are given in Example \ref{cohom-zero-F_1}.

Let $a$ be an integer. Then $X$ is also a projective bundle \cite{SW90}
\begin{equation}\label{eq:projective bundle 3}
X\cong \mathbb{P}(\mathcal{W}\otimes \mathcal{O}_{\mathbb{P}^2}(a)) \stackrel{\rho}\to \mathbb{P}^2,
\end{equation}
where the rank two vector bundle $\mathcal{W}$ over $\mathbb{P}^2$ is given by
\begin{equation*}
0\to \mathcal{O}_{\mathbb{P}^2}(-1)^{\oplus 2} \to \mathcal{O}_{\mathbb{P}^2}^{\oplus 4} \to \mathcal{W}(1)\to 0.
\end{equation*}

\subsection{Cohomologically zero line bundles}

\begin{lem}\label{Char-H0-H3-cubic}
$H^{0}(\mathcal{O}_{X}(aH+bE))=0$ if and only if $a<0$ or $a+2b<0$.
Consequently,
$H^{3}(\mathcal{O}_{X}(aH+bE))=0$ if and only if $a>-4$ or $a+2b>-2$.
\end{lem}

\begin{proof}
Similar to Lemma \ref{Char-H0-H3-point}, 
it suffices to show that if $aH+bE$ ($a, b\in \mathbb{Z}$) is an effective divisor, 
then $a\geq 0$ and $a+2b\geq0$.
Since $H$ is a nef divisor,
we obtain the intersection number
$$
H^{2}(aH+bE)=a\geq 0.
$$
Since $2H-E$ is base-point free and hence are nef divisors,
we have the intersection number
$$
(2H-E)^2(aH+bE)=a+2b\geq 0.
$$

Next, by Serre duality, we have
$$
H^{3}(\mathcal{O}_{X}(aH+bE))\cong H^{0}(\mathcal{O}_{X}((-4-a)H-(b-1)E)).
$$
Note that $(-4-a)-2(b-1)=-2-a-2b$ and then the proposition follows.
\end{proof}

Let $Q$ be the proper transform of a smooth quadric surface containing $C$.
Then $\mathcal{O}_{X}(Q)\cong\mathcal{O}_{X}(2H-E)$.
Since $Q\cong \mathbb{P}^1\times \mathbb{P}^1$, we may assume $\mathrm{Pic}(Q)=\mathbb{Z}C_1\oplus \mathbb{Z}C_2$, where $C_1, C_2$ are smooth rational curves on $Q$ with $C_1^2=C_2^2=0$ and $C_1C_2=1$.

\begin{lem}\label{divisor-Q-lem}
$H|_{Q}\sim C_1+C_2$ and $(2H-E)|_{Q}\sim C_1$ or $(2H-E)|_{Q}\sim C_2$.
\end{lem}

\begin{proof}
Suppose $H|_{Q}\sim aC_1+bC_2$ are linear equivalent. 
Since $H|_{Q}$ is a nef divisor, then $a, b\geq 0$.
So we get the intersection number
$$
2ab=(H|_{Q})^2=H^2(2H-E)=2H^3-H^2E=2,
$$
and thus $a=b=1$.

Suppose $(2H-E)|_{Q}\sim a'C_1+b'C_2$ are linear equivalent. 
Since $(2H-E)|_{Q}$ is a nef divisor, then $a', b'\geq 0$.
We have the intersection number
$$
a'+b'=(H|_{Q})((2H-E)|_{Q})=H(2H-E)(2H-E)=1,
$$
and hence $a'=1,b'=0$ or $a'=0,b'=1$.
\end{proof}

\begin{rem}
From now on, we may assume $(2H-E)|_{Q}\sim C_2$.
\end{rem}

For the case of the blow-up of $\mathbb{P}^3$ at a twisted cubic curve, there is no similar result as Lemma \ref{H1H2-zero-lem-point} and Lemma \ref{H1H2-zero-lem-line}. In fact, this is the reason why we do not know whether $\mathcal{O}_X(aH+bE)$ is cohomologically zero in either of the two cases (10) and (11) in Proposition \ref{cohom-zero-lem-cubic}.
But there are still some characterizations of cohomologically zero line bundles.
In fact, we know that if a line bundles $\mathcal{O}_{X}(aH+bE)$ is cohomologically zero then $\chi(\mathcal{O}_{X}(aH+bE))=0$.
Combining with Lemma \ref{Char-H0-H3-cubic}, 
we obtain that a line bundle is cohomologically zero if it satisfy at least one of the conditions in the following proposition.

\begin{prop}\label{cohom-zero-lem-cubic}
A line bundle $\mathcal{O}_{X}(aH+bE)$ is cohomologically zero if it is one of the following:
\begin{enumerate}
\item[(1)] $a+2b=-1$;
\item[(2)] $a=-1, b=1$;
\item[(3)] $a=-2, b=0$;
\item[(4)] $a=-2, b=1$;
\item[(5)] $a=-3, b=0$;
\item[(6)] $a=-4, b=2$;
\item[(7)] $a=-7, b=4$;
\item[(8)] $a=0, b=-1$;
\item[(9)] $a=3, b=-3$.
\end{enumerate}
On the other hand, if $\mathcal{O}_{X}(aH+bE)$ is cohomologically zero and belongs to none of the nine cases above, then one of the following two statements is true:
\begin{enumerate}
\item[(10)]  $a<-3, a+2b>3$ and $f(a,b):=a^2+5a+6-2ab-5b^2+b=0$;
\item[(11)]  $a>-1, a+2b<-3$ and $f(a,b)=0$.
\end{enumerate}
\end{prop}

\begin{proof}
At first, by Lemma \ref{Char-H0-H3-cubic},
we have $H^{0}(\mathcal{O}_{X}(aH+bE))=H^{3}(\mathcal{O}_{X}(aH+bE))=0$
if and only if one of the following conditions hold:
\begin{enumerate}
\item[(i)] $-4<a<0$;
\item[(ii)] $a+2b>-2, a<0$;
\item[(iii)] $a>-4, a+2b<0$;
\item[(iv)] $-2<a+2b<0$.
\end{enumerate}
By Blow-up formula of Chern classes, 
we have
$$
c_2(X)
=\pi^{\ast}(c_2(\mathbb{P}^3)+\mu_{C})-\pi^{\ast}c_{1}(\mathbb{P}^3)E
=9H^2-4HE,
$$
where $\mu_{C}=3H^2$.
By Riemann-Roch formula \eqref{RR-formula},
we have
\begin{eqnarray*}
\chi(\mathcal{O}_{X}(aH+bE)) 
&=&  \frac{1}{6}(a^3+6a^2+11a+6-10b^3-3b^2+13b+9ab-9ab^2)\\
&=& \frac{1}{6}(a+2b+1) f(a,b).
\end{eqnarray*}
Therefore, a line bundle is cohomologically zero only if it lies in one of the $11$ cases.

Second, from the short exact sequence,
$$
0\to \mathcal{O}_{X}(-Q) \to \mathcal{O}_{X} \to \mathcal{O}_{Q} \to 0,
$$
we tensor with $\mathcal{O}_{X}(kQ-H)$ and  $\mathcal{O}_{X}(-2H+E)$ to obtain two exact sequences,
\begin{equation}\label{exact-sequ-for-(1)}
0\to \mathcal{O}_{X}((k-1)Q-H)\to \mathcal{O}_{X}(kQ-H)\to \mathcal{O}_{Q}(kQ-H)\to 0, \;(\forall k\in \mathbb{Z}),
\end{equation}
and
\begin{equation}\label{exact-sequ-for-(6)}
0\to \mathcal{O}_{X}(-4H+2E)\to \mathcal{O}_{X}(-2H+E)\to \mathcal{O}_{Q}(-2H+E)\to 0.
\end{equation}

To show $(1)$, it is important to note that $ \mathcal{O}_{Q}(kQ-H)= \mathcal{O}_{Q}(-C_1+(k-1)C_2)$, and then $\mathcal{O}_{Q}(kQ-H)$ is cohomologically zero on $Q\cong \mathbb{P}^1\times \mathbb{P}^1$.
Therefore, by \eqref{exact-sequ-for-(1)}, 
we obtain
$$
h^{i}(\mathcal{O}_{X}(kQ-H))=h^{i}(\mathcal{O}_{X}((k-1)Q-H))=\cdots=h^{i}(\mathcal{O}_{X}(-H))=0, \forall k, i\in \mathbb{Z}.
$$
Hence, if $a+2b=-1$ then $H^{i}(\mathcal{O}_{X}(aH+bE))=H^{i}(\mathcal{O}_{X}(-b(2H-E)-H))=0$ 
for all $i \in \mathbb{Z}$, that is, $\mathcal{O}_{X}(aH+bE)$ is cohomologically zero.

To prove $(6)$, from Lemma \ref{divisor-Q-lem} and Example \ref{cohom-zero-F_1}, 
we first notice that the line bundle $\mathcal{O}_{Q}(-2H+E)\cong \mathcal{O}_{Q}(-C_2)$ is cohomologically zero.
From $(4)$, we will see that the line bundle $\mathcal{O}_{X}(-2H+E)$ is cohomologically zero.
By \eqref{exact-sequ-for-(6)}, we get that $\mathcal{O}_{X}(-4H+2E)$ is cohomologically zero.

Third, it is not difficult to see that if $(3)$, $(5)$ and $(8)$ hold 
then $\mathcal{O}_{X}(aH+bE)$ is cohomologically zero.

At last, from the short exact sequence,
$$
0\to \mathcal{O}_{X}(-E) \to \mathcal{O}_{X} \to \mathcal{O}_{E} \to 0,
$$
we tensor with $\mathcal{O}_{X}(-H+E)$, $\mathcal{O}_{X}(-2H+E)$, $\mathcal{O}_{X}(-7H+4E)$ and $\mathcal{O}_{X}(3H-2E)$ respectively to obtain the following exact sequences,
\begin{equation}\label{exact-sequ-for-(2)}
0\to \mathcal{O}_{X}(-H)\to \mathcal{O}_{X}(-H+E)\to \mathcal{O}_{E}(-H+E)\to 0,
\end{equation}
\begin{equation}\label{exact-sequ-for-(4)}
0\to \mathcal{O}_{X}(-2H)\to \mathcal{O}_{X}(-2H+E)\to \mathcal{O}_{E}(-2H+E)\to 0,
\end{equation}
\begin{equation}\label{exact-sequ-for-(7)}
0\to \mathcal{O}_{X}(-7H+3E)\to \mathcal{O}_{X}(-7H+4E)\to \mathcal{O}_{E}(-7H+4E)\to 0,
\end{equation}
\begin{equation}\label{exact-sequ-for-(9)}
0\to \mathcal{O}_{X}(3H-3E)\to \mathcal{O}_{X}(3H-2E)\to \mathcal{O}_{E}(3H-2E)\to 0.
\end{equation}
From Example \ref{cohom-zero-F_1}, we obtain that the line bundles,
$\mathcal{O}_{E}(-H+E)\cong \mathcal{O}_{E}(-S+2F)$,
$\mathcal{O}_{E}(-2H+E)\cong \mathcal{O}_{E}(-S-F)$,
$\mathcal{O}_{E}(-7H+4E)\cong \mathcal{O}_{E}(-4S-F)$
and 
$\mathcal{O}_{E}(3H-2E)\cong \mathcal{O}_{E}(2S-F)$
are cohomologically zero on $E\cong \mathbb{P}^1\times \mathbb{P}^1$.
Since $3+2\times (-2)=-1$ and $-7+2\times 3=-1$,
from $(1)$, we get that $\mathcal{O}_{X}(3H-2E)$ and $\mathcal{O}_{X}(-7H+3E)$ are cohomologically zero.
Then $(2)$, $(4)$, $(7)$ and $(9)$ hold following the exact sequences \eqref{exact-sequ-for-(2)} \eqref{exact-sequ-for-(4)}, \eqref{exact-sequ-for-(7)} and \eqref{exact-sequ-for-(9)} respectively.
\end{proof}

\subsection{Classification results}

\begin{thm}\label{Classify-Blowup-twisted-cubic-P3}
Let $X$ be the blow-up of $\mathbb{P}^{3}$ at a twisted cubic curve.
Then the normalized sequence
\begin{equation}\label{basic-type-of-EFC-linebundle3}
\{\mathcal{O}_{X},\mathcal{O}_{X}(D_1), \mathcal{O}_{X}(D_2),\mathcal{O}_{X}(D_3), \mathcal{O}_{X}(D_4), \mathcal{O}_{X}(D_5)\}
\end{equation}
is an exceptional collection of line bundles  
if and only if the ordered set of divisors $\{D_1, D_2, D_3, D_4, D_5\}$ is one of the following types:
\begin{enumerate}
 \item[$(1)$] $\{H, 3H-E, 
                 E, 2H, 3H \}$;

\item[$(2)$] $\{2H-E, -H+E, 
                 H, 2H, 3H-E\}$; 

\item[$(3)$] $\{-3H+2E, -H+E, 
                 E, H, 2H\}$; 

\item[$(4)$] $\{2H-E, 3H-E, 
                 4H-2E, 5H-2E, 7H-3E\}$;                     

\item[$(5)$] $\{H, 2H-E, 
                 3H-E, 5H-2E, 2H \}$;

\item[$(6)$] $\{H-E, 2H-E, 
                 4H-2E, H, 3H-E \}$;
                       
\item[$(7)$] $\{2H-E, -3H+2E, 
                 4H-2E, -H+E, H \}$;                

\item[$(8)$]  $\{-5H+3E, 2H-E, 
                   -3H+2E, -H+E, 2H \}$; 

\item[$(9)$]  $\{7H-4E, 2H-E, 
                   4H-2E, 7H-3E, 9H-4E \}$; 

\item[$(10)$] $\{-5H+3E, -3H+2E, 
                  E, 2H, -3H+3E \}$;   

\item[$(11)$] $\{2H-E, 5H-2E, 
                    7H-3E, 2H, 9H-4E\}$;                
                                               
\item[$(12)$] $\{3H-E, 5H-2E, 
                    E, 7H-3E, 2H\}$;  
                                            
\item[$(13)_b$]   $\{2H-E, 4H-2E, 
                      (2b-3)H-(b-2)E,  
                      (2b-1)H-(b-1)E,  (2b+1)H-bE \}$;

\item[$(14)_b$]   $\{2H-E,  
                      (2b-3)H-(b-2)E, (2b-1)H-(b-1)E,  
                      (2b+1)H-bE, 2H \}$;
               
\item[$(15)_b$]   $\{(2b-3)H-(b-2)E, 
                     (2b-1)H-(b-1)E,  (2b+1)H-bE, 
                     E, 2H \}$,
\end{enumerate}
where $b\in \mathbb{Z}$. Moreover, by mutations and normalizations, they are related as:
\begin{eqnarray*}
& & (1)\Rightarrow(2)\Rightarrow(3)\Rightarrow(4)\Rightarrow(5)\Rightarrow (6) \Rightarrow (1); \\
& & (7)\Rightarrow(8)\Rightarrow(9)\Rightarrow(10)\Rightarrow(11)\Rightarrow (12)
\Rightarrow (7);\\
& & (13)_b\Rightarrow(14)_{b-1}\Rightarrow(15)_{b-2}\Rightarrow (13)_{5-b}.
\end{eqnarray*}
\end{thm}

\begin{proof}
The idea of the proof is the same as that of Theorem \ref{Classify-Blowup-p-P3}. Write $D_0=0$. By Lemma \ref{lem:exlb},
the sequence \eqref{basic-type-of-EFC-linebundle3}
is an exceptional collection if and only if for any integers $0\leq j<i\leq 5$ the line bundles $\mathcal{O}_{X}(D_j-D_i)$ are cohomologically zero.

Suppose the sequence \eqref{basic-type-of-EFC-linebundle3} is an exceptional collection.
By Proposition \ref{cohom-zero-lem-cubic},
the line bundle $\mathcal{O}_{X}(D_{i})$ must be one of the following line bundles:
$B_0=\mathcal{O}_{X}((2b+1)H-bE)$,
$B_1=\mathcal{O}_{X}(H-E)$,
$B_2=\mathcal{O}_{X}(2H)$,
$B_3=\mathcal{O}_{X}(2H-E)$,
$B_4=\mathcal{O}_{X}(3H)$,
$B_5=\mathcal{O}_{X}(4H-2E)$,
$B_6=\mathcal{O}_{X}(7H-4E)$,
$B_7=\mathcal{O}_{X}(E)$
$B_8=\mathcal{O}_{X}(-3H+3E)$,
$B_9=\mathcal{O}_{X}(aH+bE)$ ($a, b$ satisfy condition (10) in Proposition \ref{cohom-zero-lem-cubic})
and $B_{10}=\mathcal{O}_{X}(aH+bE)$ ($a, b$ satisfy condition (11) in Proposition \ref{cohom-zero-lem-cubic}).
Furthermore,  
to find out all the exceptional collections \eqref{basic-type-of-EFC-linebundle3}, 
it suffices to pick up $\{B_{l_1},\cdots ,B_{l_5}\}$ such that each pair $(B_{l_j},B_{l_i})$ ($j<i$) is an exceptional pair.
To attain this, we shall build up a table of all exceptional pairs which consists of $B_i$ ($i=0,1, \cdots,10$).
Since a pair $(B_s, B_t)$ is an exceptional pair if and only if the line bundle  $B_s \otimes B_t^{\vee}$ is cohomologically zero, by Proposition \ref{cohom-zero-lem-cubic} we get the following table:

\begin{table}[ht]
\caption{Exceptional pairs $(B_s, B_t)$ for blow-up of $\mathbb{P}^3$ at a twisted cubic curve}\label{tab3:g}
\begin{center}
{\tiny

\begin{tabular}{|p{3mm}|p{15mm}|p{3mm}|p{3mm}|p{12mm}|p{12mm}|p{12mm}|p{3mm}|p{3mm}|p{12mm}|p{3mm}|p{3mm}|}

\hline 
  & $B_0^{\prime}$ & $B_1$ & $B_2$ & $B_3$ & $B_4$ & $B_5$ & $B_6$ & $B_7$ & $B_8 $& $B_9$ & $B_{10}$\\
\hline 
 $B_0$& $b=b^{\prime}+1$, $b^{\prime}+2$& &$\forall b$&$b=0$, 3& $b=0$, $-1$&$b=-1$,  $2$& &$\forall b$ &$b=2$, 3& & \\
\hline 
 $B_1$& $b^{\prime}=0$, $-1$& & &$\surd$& &$\surd$& & & & & \\
\hline 
 $B_2$& $b^{\prime}=-1$, $-4$& & & & $\surd$& & & &$\surd$& & \\
\hline 
 $B_3$&$\forall  b^{\prime}$& &$\surd$ & & &$\surd$& & & & & \\
\hline 

 $B_4$& & & & & & & & & & & \\
 \hline
 $B_5$& $\forall  b^{\prime}$& & & & & & & & & & \\
 \hline
 $B_6$& $b^{\prime}=-3$, $-4$& & &$\surd$& &$\surd$& & & & & \\
 \hline
 $B_7$& $b^{\prime}=0$, $-3$& & $\surd$ & &$\surd$ & & & &$\surd$ & & \\
 \hline
 $B_8$& & & & & & & & & & & \\
 \hline
 $B_9$&  & & & & & & & & &? &? \\
 \hline
 $B_{10}$& & & & & & & & & & ?& ?\\
 \hline
 \end{tabular}
}
\end{center}
\;
\;
\small{
The mark ``?" means that we do not know whether the corresponding pair could be an exceptional pair or not.}

\end{table}

\begin{claim}\label{small-lem2}
$\{\mathcal{O}_{X}, \mathcal{O}_{X}((2b_1+1)H-b_1E), \cdots, \mathcal{O}_{X}((2b_i+1)H-b_iE)\}$ is an exceptional collection if and only if it is one of the following conditions hold:
\begin{enumerate}
\item[(1)] $i=1$, $b_1\in \mathbb{Z}$;

\item[(2)] $i=2$, $b_1=b_2-1$ or $b_1=b_2-2$;

\item[(3)] $i=3$, $b_1+1=b_2=b_3-1$.
\end{enumerate}
\end{claim}

\begin{proof}[Proof of Claim \ref{small-lem2} ]
The pair $\{\mathcal{O}_{X}((2a+1)H-aE), \mathcal{O}_{X}((2b+1)H-bE)\}$ is an exception collection if and only if $\mathcal{O}_{X}(2(a-b)H-(a-b)E)$ is cohomologically zero. 
Then, by Proposition \ref{cohom-zero-lem-cubic} (4) and (6), 
we get $a-b=-1$ or $a-b=-2$ and the claim follows.
\end{proof}

\begin{claim}\label{no-(10-11)}
Let $D_i=a_i H+b_i E$, $i=1,2,3$, be three divisors on $X$. We assume, for any $1\le i\le 3$, $-D_i$ satisfies either (10) or (11) in Proposition \ref{cohom-zero-lem-cubic}. Then $(D_1, D_2, D_3)$ cannot be an exceptional collection of length 3.
\end{claim}

\begin{proof}[Proof of Claim \ref{no-(10-11)}]  
Denote $g(a,b):=\frac{1}{6}(a+2b+1) f(a,b)$, which is the Euler characteristic function in the proof of Proposition \ref{cohom-zero-lem-cubic}. Here $f(a,b):=a^2+5a+6-2ab-5b^2+b$.
The following system of equations 
\begin{eqnarray*}
f(-a_1, -b_1) = f(-a_2, -b_2) = f(-a_3, -b_3) =0,\\
g(a_1 - a_2, b_1 - b_2) =g(a_1 - a_3, b_1 - b_3)=g(a_2 - a_3, b_2 - b_3) =0,
\end{eqnarray*}
has exactly four solutions (e.g., by computer software \texttt{Mathematica}): 
$$
(a_1, b_1, a_2, b_2, a_3, b_3)=(0,1,2,0,-3,3), (0,1,2,0,3,0), (1,-1,2,-1,4,-2), \text{ or }(7,-4,2,-1,4,-2). 
$$
Then the claim follows.
\end{proof}

From Claim \ref{small-lem2}, Claim \ref{no-(10-11)} and Table \ref{tab3:g}, 
we observe that the line bundle $\mathcal{O}_{X}(D_1)$ in the sequence \eqref{basic-type-of-EFC-linebundle3} may be only one of $B_0$, $B_1$, $B_3$ and $B_6$.
Similar to the proof of Theorem \ref{Classify-Blowup-p-P3}, 
one may discuss $\mathcal{O}_{X}(D_1)$ case-by-case to obtain Theorem \ref{Classify-Blowup-twisted-cubic-P3}. 
Since this is totally the same as the proof of Theorem \ref{Classify-Blowup-twisted-cubic-P3}, we leave it to the interested readers.
\end{proof}


Now we are in the position to prove the main result of the current section.

\begin{thm}\label{main-thm-cubic}
Let $X$ be the blow-up of $\mathbb{P}^{3}$ at a twisted cubic curve.
Then any exceptional collection of line bundles of length $6$ 
on $X$ is full.
\end{thm}

\begin{proof}
By Lemma \ref{normalization-lem} and Lemma \ref{mutation-fullness-lem},
it suffices to show the exceptional collection of type $(1)$, $(7)$ and $(13)$ in Theorem \ref{Classify-Blowup-twisted-cubic-P3} are full. Then the proof is divided into three parts.

\textbf{Part 1}: Type $(1)$, i.e. $\{\mathcal{O}_{X}, \mathcal{O}_{X}(H), \mathcal{O}_{X}(3H-E), 
\mathcal{O}_{X}(E), \mathcal{O}_{X}(2H), \mathcal{O}_{X}(3H) \}$ is a full exceptional collection.

Applying Orlov's blow-up formula (Theorem \ref{Orlov-blowup-formula}) to the blow-up $X$,
we obtain a semiorthogonal decomposition of $\mathrm{D}(X)$,
\begin{equation*}
\langle 
j_{\ast}\rho^{\ast}\mathrm{D}(C)\otimes \mathcal{O}_{E}(-1),
\mathcal{O}_{X}, \mathcal{O}_{X}(H), \mathcal{O}_{X}(2H), \mathcal{O}_{X}(3H)
\rangle.
\end{equation*}
Since $\mathcal{O}_{E}(E)\cong \mathcal{O}_{E}(-1,5)$,
this turns to be 
$$
\langle 
\mathcal{O}_{E}(-1,4), \mathcal{O}_{E}(-1,5),
\mathcal{O}_{X}, \mathcal{O}_{X}(H), \mathcal{O}_{X}(2H), \mathcal{O}_{X}(3H)
\rangle.
$$
From the short exact sequence,
$$
0\to \mathcal{O}_{X}(-E) \to \mathcal{O}_{X} \to \mathcal{O}_{E} \to 0,
$$
we tensor with $\mathcal{O}_{X}(E)$ to gain an exact sequence,
$$
0\to \mathcal{O}_{X}\to \mathcal{O}_{X}(E)\to \mathcal{O}_{E}(E)\to 0.
$$
Hence
$\mathcal{O}_{E}(-1,5)\cong \mathcal{O}_{E}(E)\in\langle 
\mathcal{O}_{X}, \mathcal{O}_{X}(H), \mathcal{O}_{X}(3H-E), 
\mathcal{O}_{X}(E), \mathcal{O}_{X}(2H), \mathcal{O}_{X}(3H) 
\rangle.
$
Since 
$$
\mathrm{Hom}(\mathcal{O}_{X}(3H-E), \mathcal{O}_{E}(-1,4)[k])
=H^{k}(\mathcal{O}_{E}(-2,0))
=
\left\{
 \begin{array}{ll}
  0, k\neq 0; \\
   \mathbb{C}, k=1;
 \end{array},
\right.
$$
hence the pair $(\mathcal{O}_{E}(-1,4), \mathcal{O}_{X}(3H-E))$ is not an exceptional pair,
and then the result follows from the same idea as the proof of Lemma \ref{usefull-fullness-lem}.

\;
\;
\;

\textbf{Part 2}: 
We start with the full exceptional collection of type $(6)$ in Theorem \ref{Classify-Blowup-twisted-cubic-P3},
$$
\mathrm{D}(X)
=\langle\mathcal{O}_{X}, \mathcal{O}_{X}(H-E), \mathcal{O}_{X}(2H-E), 
\mathcal{O}_{X}(4H-2E), \mathcal{O}_{X}(H), \mathcal{O}_{X}(3H-E)\rangle.
$$
Since $\mathcal{O}_{X}(H-E)$ is cohomologically zero,
mutating $\mathcal{O}_{X}(H-E)$ and $\mathcal{O}_{X}$ turns this into
\begin{equation}\label{completion-mutation-sod}
\mathrm{D}(X)
=\langle\mathcal{O}_{X}(H-E), \mathcal{O}_{X}, \mathcal{O}_{X}(2H-E), 
\mathcal{O}_{X}(4H-2E), \mathcal{O}_{X}(H), \mathcal{O}_{X}(3H-E)\rangle.
\end{equation}
Then, by Lemma \ref{mutation-fullness-lem}, we left mutate the last term to the front of the exceptional collection,
hence tensoring it with $\omega_X=\mathcal{O}_{X}(-4H+E)$ to obtain a semiorthogonal decomposition of  $\mathrm{D}(X)$,
$$
\langle\mathcal{O}_{X}(-H), \mathcal{O}_{X}(H-E), \mathcal{O}_{X}, \mathcal{O}_{X}(2H-E), 
\mathcal{O}_{X}(4H-2E), \mathcal{O}_{X}(H)\rangle.
$$

Next we  show that type $(7)$, i.e. $\{\mathcal{O}_{X}, \mathcal{O}_{X}(2H-E), \mathcal{O}_{X}(-3H+2E), 
 \mathcal{O}_{X}(4H-2E), \mathcal{O}_{X}(-H+E), \mathcal{O}_{X}(H) \}$ 
 is a full exceptional collection.
We assume 
\begin{equation}\label{assumption-cubic-part2}
A\in \langle \mathcal{O}_{X}, \mathcal{O}_{X}(2H-E), \mathcal{O}_{X}(-3H+2E), 
 \mathcal{O}_{X}(4H-2E), \mathcal{O}_{X}(-H+E), \mathcal{O}_{X}(H) \rangle^{\bot}
\end{equation}
and $A\neq 0$. 
Then $A\in \langle \mathcal{O}_{X}, \mathcal{O}_{X}(2H-E), 
\mathcal{O}_{X}(4H-2E), \mathcal{O}_{X}(H)\rangle^{\bot}$.
Since $\langle \mathcal{O}_{X}(-H), \mathcal{O}_{X}(H-E) \rangle 
\cong \langle \mathcal{O}_{X}, \mathcal{O}_{X}(2H-E), 
\mathcal{O}_{X}(4H-2E), \mathcal{O}_{X}(H) \rangle^{\bot}$,
thus $A\in \langle \mathcal{O}_{X}(-H), \mathcal{O}_{X}(H-E)\rangle$.
Then there is a distinguished triangle
\begin{equation}\label{cubic-exact-triangle}
A_2 \to A \to A_1 \to A_2[1],
\end{equation}
where $A_2 \in \langle \mathcal{O}_{X}(H-E)\rangle$ 
and $A_1 \in \langle  \mathcal{O}_{X}(-H)\rangle$.
Here we may assume $A_1=\bigoplus \mathcal{O}_{X}(-H)[i]^{\oplus j_i}$ 
and  $A_2=\bigoplus \mathcal{O}_{X}(H-E)[s]^{\oplus t_s}$.
By Proposition \ref{cohom-zero-lem-cubic},  
$\mathcal{O}_{X}(2H-2E)$ is not cohomologically zero, 
i.e., for some $k_0$, $H^{k_0}(\mathcal{O}_{X}(2H-2E))\neq0$.
Then we have following:
\begin{enumerate}
\item[(i)] If $A_2=0$ , then $A\cong A_1\in \langle  \mathcal{O}_{X}(-H) \rangle$ and
    $$
    \mathrm{Hom}(\mathcal{O}_{X}(-3H+2E), \mathcal{O}_{X}(-H)[k_0])=
    H^{k_0}(\mathcal{O}_{X}(2H-2E))\neq 0.
    $$
    This is contradicting to the assumption \eqref{assumption-cubic-part2}.
    
\item[(ii)] If $A_1=0$ , then $A\cong A_2\in \langle  \mathcal{O}_{X}(H-E) \rangle$, and
    $$
    \mathrm{Hom}(\mathcal{O}_{X}(-H+E), \mathcal{O}_{X}(H-E)[k_0])=
    H^{k_0}(\mathcal{O}_{X}(2H-2E))\neq 0.
    $$ 
    This is also contradicting to the assumption \eqref{assumption-cubic-part2}.

\item[(iii)]
Assume $A_1\neq 0$ and $A_2\neq0$, and 
pick up an object $B=\bigoplus \mathcal{O}_{X}(-H+E)[s]^{\oplus t_s}$.
Applying the functor $\mathrm{Hom}(B, -)$ to distinguished triangle \eqref{cubic-exact-triangle}, we gain the following long exact sequence 
$$
\cdots \to \mathrm{Hom}(B, A[k])  \to \mathrm{Hom}(B, A_1[k]) \to \mathrm{Hom}(B, A_2[k+1]) \to \mathrm{Hom}(B, A[k+1])\to \cdots .
$$
By assumption \eqref{assumption-cubic-part2}, 
this exact sequence implies that 
\begin{equation}\label{hom-exact-seq}
\mathrm{Hom}(B, A_1[k]) \cong \mathrm{Hom}(B, A_2[k+1])
\end{equation}
for any $k\in \mathbb{Z}$.
However,  
since  $\mathcal{O}_{X}(-E)$ is cohomologically zero,
for any $k\in \mathbb{Z}$, we have
$$
\mathrm{Hom}(B, A_1[k])
=\mathrm{Hom}(\bigoplus \mathcal{O}_{X}(-H+E)[s]^{\oplus t_s}, \bigoplus \mathcal{O}_{X}(-H)[i]^{\oplus j_i}[k])
=0,
$$ 
but
$$
\mathrm{Hom}(B, A_2[k_0-1])=
\mathrm{Hom}(\bigoplus \mathcal{O}_{X}(-H+E)[s]^{\oplus t_s}, \bigoplus \mathcal{O}_{X}(H-E)[s]^{\oplus t_s}[k_0])\neq 0
$$
gives a contradiction to \eqref{hom-exact-seq}.
\end{enumerate}

\;
\;
\;

\textbf{Part 3}: Type $(13)$, i.e. $\{\mathcal{O}_{X}, \mathcal{O}_{X}(2H-E), \mathcal{O}_{X}(4H-2E), 
\mathcal{O}_{X}((2b-3)H-(b-2)E), \mathcal{O}_{X}((2b-1)H-(b-1)E), 
 \mathcal{O}_{X}((2b+1)H-bE) \}$ is a full exceptional collection.

First, by Lemma \ref{mutation-fullness-lem}, we right mutate the first term to the end of the semiorthogonal decomposition \eqref{completion-mutation-sod},
hence tensoring it with $\omega_X^{\vee}=\mathcal{O}_{X}(4H-E)$ to obtain a new semiorthogonal decomposition of  $\mathrm{D}(X)$,
$$
\langle
 \mathcal{O}_{X}, \mathcal{O}_{X}(2H-E),  \mathcal{O}_{X}(4H-2E), \mathcal{O}_{X}(H), \mathcal{O}_{X}(3H-E), \mathcal{O}_{X}(5H-2E)\rangle.
$$
Therefore, for any given $a\in \mathbb{Z}$, 
inductively on $b\in \mathbb{Z}$, 
we may assume that for $b=k$ the exceptional collection
\begin{eqnarray}\label{main-thm-cubic-equ1}
&& \{\mathcal{O}_{X}, \mathcal{O}_{X}(2H-E), \mathcal{O}_{X}(4H-2E), 
      \mathcal{O}_{X}((2k-3)H-(k-2)E),  \nonumber\\
&& \;\;\;\; \mathcal{O}_{X}((2k-1)H-(k-1)E),  \mathcal{O}_{X}((2k+1)H-kE) \}
\end{eqnarray}
is full.
Then for $b=k-1$ we have the following exceptional collection
\begin{eqnarray}\label{main-thm-cubic-equ2}
&& \{\mathcal{O}_{X}, \mathcal{O}_{X}(2H-E), \mathcal{O}_{X}(4H-2E), 
      \mathcal{O}_{X}((2k-5)H-(k-3)E), \nonumber\\
&& \;\;\;\;    \mathcal{O}_{X}((2k-3)H-(k-2)E), \mathcal{O}_{X}((2k-1)H-(k-1)E) \},
\end{eqnarray}
and for $b=k+1$ we obtain the exceptional collection
\begin{eqnarray}\label{main-thm-cubic-equ3}
&& \{\mathcal{O}_{X}, \mathcal{O}_{X}(2H-E), \mathcal{O}_{X}(4H-2E), 
      \mathcal{O}_{X}((2k-1)H-(k-1)E), \nonumber\\
&& \;\;\;\;    \mathcal{O}_{X}((2k+1)H-kE), \mathcal{O}_{X}((2k+3)H-(k+1)E) \},
\end{eqnarray}
By comparing two exceptional collections \eqref{main-thm-cubic-equ2} and \eqref{main-thm-cubic-equ3} with \eqref{main-thm-cubic-equ1},
Lemma \ref{usefull-fullness-lem} implies that the exceptional collections \eqref{main-thm-cubic-equ2} and \eqref{main-thm-cubic-equ3} are full.
This completes the proof of Theorem \ref{main-thm-cubic}.
\end{proof}

\begin{rem}
It is possible to show the fullness of type $(13)$ by using the projective bundle structure (\ref{eq:projective bundle 3}) and projective bundle formula.
\end{rem}

\begin{rem}
Since the blow-up a point (or a line, a twisted cubic curve) of $\mathbb{P}^3$ is a Fano variety, theoretically one may use Bondal \cite[Theorem 4.1]{Bon90} to show the fullness of the exceptional collections in Theorem \ref{Classify-Blowup-p-P3}, Theorem \ref{Classify-Blowup-line-P3} and Theorem \ref{Classify-Blowup-twisted-cubic-P3}.
\end{rem}


\section{Final remarks}

In this section, we collect some interesting problems which are related to Kuznetsov's fullness conjecture and which we are interested in.

(1) 
Recall that all smooth toric Fano $3$-folds and $4$-folds have full exceptional collections of line bundles. 
Inspired by main theorem, one may hope the following to be true.

\begin{conj}
All exceptional collections of line bundles of maximal length on smooth toric Fano $3$-folds and $4$-folds are full.
\end{conj}

Of course, it is also very interesting to give an explicit classification of full exceptional collections of line bundles on smooth toric Fano $3$-folds and $4$-folds.

In dimensional $4$, there are some relatively easy examples of smooth toric Fano $4$-folds on which cohomologically zero line bundles are clear:
\begin{enumerate} 
\item[(i)] $\mathcal{O}_{\mathbb{P}^2\times \mathbb{P}^2}(a, b)$ is cohomologically zero if and only if $a=-1, -2$, or $b=-1, -2$;

\item[(ii)] $\mathcal{O}_{\mathbb{P}^3\times \mathbb{P}^1}(a, b)$ is cohomologically zero if and only if $a=-1, -2, -3$, or $b=-1$;

\item[(iii)] $\mathcal{O}_{\mathbb{P}^1\times \mathbb{P}^1 \times \mathbb{P}^1\times \mathbb{P}^1}(a, b, c)$ is cohomologically zero if and only if $a=-1$, or $b=-1$, or $c=-1$.
\end{enumerate} 
In general, the classification of cohomologically zero line bundles on smooth toric Fano 4-folds may be much more complicated.
For example, it is interesting to give the classification of cohomologically zero line bundles on the blow-up of $\mathbb{P}^{4}$ at a point or a line.

(2) 
The existence of (quasi-)phantom categories on smooth projective varieties is becoming very important.  There is an interesting conjecture on the existence of phantom categories on Barlow surfaces (see \cite[Conjecture 4.1]{DKK13} and \cite[Conjecture 4.9]{CKP13}), which is strongly connected to Kuznetsov's fullness conjecture. 

\begin{conj}\label{Barlow-conj}
Given any Barlow surface $X$, 
the derived category $\mathrm{D}(X)$ has an
exceptional collection of length $11$ with orthogonal complement a (non-trivial) phantom category,
but has no full exceptional collection.
\end{conj}

We see that Kuznetsov's fullness conjecture implies Conjecture \ref{Barlow-conj} for some Barlow surfaces.
As a matter of fact, in \cite{BGvBKS15}, B\"{a}hning-Graf von Bothmer-Katzarkov-Sosna
have exhibited a determinantal Barlow surface $S$
with an exceptional collection of line bundles of length $11$ whose orthogonal complement is a (non-trivial) phantom category.
Certainly, this does not mean that there is no full exceptional collection on the Barlow surface $S$.
If $S$ admits a full exceptional collection then its length must be $11$,
and then Kuznetsov's fullness conjecture give a contradiction with B\"{a}hning-Graf von Bothmer-Katzarkov-Sosna's result.

(3) 
Since any smooth projective rational surface is a series of blow-up points of $\mathbb{P}^2$ or Hirzebruch surfaces, hence Orlov's blow-up formula implies that any smooth projective rational surface admits a full exceptional collection. Moreover, the augmentation implies that any smooth projective rational surface admits a full exceptional collection of line bundles. 
Conversely, there is an open problem: 

\begin{conj}[Orlov]\label{surf-rational-excep-collect}
Any smooth projective surface with a full exceptional collection (of line bundles) is rational.
\end{conj}

As a byproduct, Conjecture \ref{surf-rational-excep-collect} implies that any smooth projective surface with a full exceptional collection of length $4$ must be a Hirzebruch surface.
Recently, in \cite[Theorem 4.3]{BS15}, Brown-Shipman show that Conjecture \ref{surf-rational-excep-collect} holds for a smooth projective surface which admits a full strong exceptional collection of line bundles.
However, it is still widely open in general.
Also, there is a high dimensional analogous to Conjecture \ref{surf-rational-excep-collect} (see \cite[Conjecture 1.2]{EL15}).


\end{document}